\documentclass[12pt]{article}
\usepackage{lmodern}
\usepackage{microtype}
\usepackage{enumitem}
\usepackage{natbib}
\usepackage{float}
\usepackage{amsmath,amssymb,amsfonts,amsthm,mathtools}
\usepackage{xcolor}
\usepackage{soul}
\definecolor{Rubro}{RGB}{220,20,60}
\definecolor{Emerald}{HTML}{00674F}
\sethlcolor{yellow!35}
\usepackage{hyperref}
\hypersetup{
    colorlinks=true,
    linkcolor=blue,
    citecolor=Emerald,
    filecolor=magenta,
    urlcolor=blue,
    pdftitle={Hard-Edge Determinant Fields and Second-Moment Universality for Random Matrix Ratios},
    pdfpagemode=FullScreen
}
\makeatletter
\renewcommand{\hyper@natlinkstart}[1]{%
  \Hy@backout{#1}%
  \textbf\bgroup
  \hyper@linkstart{cite}{cite.#1}%
  \def\hyper@nat@current{#1}%
}
\renewcommand{\hyper@natlinkend}{%
  \hyper@linkend
  \egroup
}
\renewcommand{\hyper@natlinkbreak}[2]{%
  \hyper@linkend
  \egroup
  #1%
  \textbf\bgroup
  \hyper@linkstart{cite}{cite.#2}%
}
\makeatother
\usepackage{booktabs}
\usepackage[margin=2.25cm]{geometry}

\soulregister\cite7
\soulregister\citet7
\soulregister\citep7
\soulregister\ref7
\soulregister\pageref7

\newtheorem{theorem}{Theorem}[section]

\newtheorem{remark}[theorem]{Remark}
\newtheorem{assumption}[theorem]{Assumption}

\newcommand{\E}{\mathbb{E}}
\renewcommand{\Pr}{\mathbb{P}}
\renewcommand{\leq}{\leqslant}
\renewcommand{\geq}{\geqslant}

\newcommand{\Var}{\mathbb{V}}

\renewcommand{\limsup}{\operatorname{limsup}}

\newcommand{\1}{\mathbf{1}}

\title{Hard-Edge Determinant Fields and Second-Moment Universality for Random Matrix Ratios}
\author{Guilherme Vianna%
\footnote{University of S\~ao Paulo/Columbia University; \texttt{guilherme.dias.vianna@usp.br}.}%
\footnote{The author thanks Professor Yuanyuan Xu (Academy of Mathematics and Systems Science, Chinese Academy of Sciences) for helpful comments on the process of drafting this paper.}}
\date{}

\begin{document}

\maketitle

\begin{abstract}
We study the microscopic spectrum at the origin of ratios of independent complex Girko matrices. Under bounded-density and finite-moment assumptions, together with circular second-moment matching, we prove compact-uniform convergence of the normalized determinants
to a random entire function constructed from the complex hard edge and an independent Ginibre array. The zero divisor of this limiting determinant field has the law of the infinite complex Ginibre process. Consequently, the first finitely many smallest and largest eigenvalue moduli, the inner and outer spectral radii, and eigenvalue counts in fixed bounded sets have universal limits depending only on the second moments. This answers the second-moment question posed by Chafaï, García-Zelada, and Xu for the spectral radii. We also obtain a multivariate determinant field for finitely many perturbation directions and prove the uniform hard-edge comparison needed to extend the conclusions to triangular arrays of atom laws.
\end{abstract}

\section{Preliminaries}
\label{sec:preliminaries}

\subsection{Matrix ratios and assumptions}
\label{subsec:model-and-assumptions}

An \(n\times n\) complex Girko matrix is a random matrix whose entries are independent and identically distributed complex random variables with mean zero and variance one. A complex Ginibre matrix is the special case in which the common entry has the form
\(
\dfrac{G_1+\mathrm{i}G_2}{\sqrt{2}},
\)
where \(G_1\) and \(G_2\) are independent standard real Gaussian random variables.

\begin{assumption}
\label{ass:matrix-atoms}
For every \(n\geq 1\), let
\[
A_n=\left(A_{n,ij}\right)_{i,j=1}^n
\qquad\text{and}\qquad
B_n=\left(B_{n,ij}\right)_{i,j=1}^n
\]
be independent complex Girko matrices. The entries of \(A_n\) are independent copies of a complex random variable \(\Xi_A\), and the entries of \(B_n\) are independent copies of a complex random variable \(\Xi_B\). The laws of \(\Xi_A\) and \(\Xi_B\) may differ, but neither law depends on \(n\). For each \(\nu\in\left\{A,B\right\}\),
\begin{equation}
\E\left[\Xi_\nu\right]=0,
\qquad
\E\left[\left|\Xi_\nu\right|^2\right]=1,
\qquad
\E\left[\Xi_\nu^2\right]=0.
\label{eq:circular-second-moments}
\end{equation}
Moreover, \(\E\left[\left|\Xi_\nu\right|^p\right]<\infty\) for every integer \(p\geq 1\), and the law of \(\Xi_\nu\) admits a Lebesgue density
\(\varphi_\nu\colon\mathbb{C}\to\left[0,\infty\right)\) satisfying
\[
\left\|\varphi_\nu\right\|_{L^\infty\left(\mathbb{C}\right)}<\infty.
\]
\end{assumption}

The identities in \eqref{eq:circular-second-moments} match the real bivariate moments of a standard complex Gaussian random variable through total degree two. Indeed, writing
\(\Xi_\nu=X_\nu+\mathrm{i}Y_\nu\), they imply
\[
\E\left[X_\nu^2\right]
=
\E\left[Y_\nu^2\right]
=
\frac{1}{2},
\qquad
\E\left[X_\nu Y_\nu\right]=0.
\]
We refer to \eqref{eq:circular-second-moments} as the circular second-moment condition.

The density assumption implies that \(A_n\) and \(B_n\) are almost surely invertible. We therefore define the matrix ratio
\[
M_n:=A_nB_n^{-1}.
\]
For a matrix \(T\in\mathbb{C}^{n\times n}\), let \(\operatorname{spec}\left(T\right)\) denote its spectrum as a multiset, with each eigenvalue repeated according to its algebraic multiplicity, and define
\[
\rho_{\min}\left(T\right)
:=
\min_{\lambda\in\operatorname{spec}\left(T\right)}
\left|\lambda\right|,
\quad
\rho_{\max}\left(T\right)
:=
\max_{\lambda\in\operatorname{spec}\left(T\right)}
\left|\lambda\right|.
\]
Since \(M_n^{-1}=B_nA_n^{-1}\), spectral inversion gives
\begin{equation}
\begin{aligned}
\operatorname{spec}\left(B_nA_n^{-1}\right)
=
\left\{
\lambda^{-1}:
\lambda\in\operatorname{spec}\left(M_n\right)
\right\},\quad
\rho_{\max}\left(M_n\right)
=
\frac{1}{\rho_{\min}\left(B_nA_n^{-1}\right)}.
\end{aligned}
\label{eq:spectral-inversion}
\end{equation}

\subsection{Notation}
\label{subsec:notation}

For \(T\in\mathbb{C}^{n\times n}\), its singular values are ordered increasingly as
\(
0\leq s_1\left(T\right)\leq\cdots\leq s_n\left(T\right).
\)

Let \(\mathcal{N}\left(\mathbb{C}\right)\) denote the space of locally finite point measures on \(\mathbb{C}\), endowed with the vague topology. Thus, \(\mu_n\to\mu\) vaguely when
\[
\int_{\mathbb{C}} g\left(z\right)\,\mu_n\left(\mathrm{d}z\right)
\rightarrow
\int_{\mathbb{C}} g\left(z\right)\,\mu\left(\mathrm{d}z\right)
\]
for every \(g\in C_c\left(\mathbb{C}\right)\). Points are counted with their multiplicities.

For \(r\geq 1\), let \(\operatorname{Hol}\left(\mathbb{C}^r\right)\) be the space of holomorphic functions on \(\mathbb{C}^r\), endowed with uniform convergence on compact sets. If \(f\in\operatorname{Hol}\left(\mathbb{C}\right)\) is not identically zero, its zero-counting measure is
\[
\mathcal{Z}\left(f\right)
:=
\sum_{\substack{z\in\mathbb{C}\\ f\left(z\right)=0}}
m_f\left(z\right)\delta_z,
\]
where \(m_f\left(z\right)\) is the multiplicity of the zero at \(z\).

The infinite complex Ginibre process, denoted by \(\operatorname{Gin}_{\infty}\), is the determinantal point process on \(\mathbb{C}\) with kernel
\[
K_{\infty}\left(z,w\right)
=
\frac{1}{\pi}
\exp\left(
z\overline{w}
-\frac{\left|z\right|^2}{2}
-\frac{\left|w\right|^2}{2}
\right)
\]
with respect to planar Lebesgue measure. We write
\[
0<R_1<R_2<\cdots
\]
for the ordered moduli of its points; the inequalities are strict almost surely.

\newpage

\section{Main results}
\label{sec:main-results}

We now state and prove the main results of the paper.

\subsection{The scalar hard-edge determinant field}
\label{subsec:determinant-field}

Since \(A_n\) is almost surely invertible, we may define
\begin{equation}
Q_n\left(w\right)
:=
\frac{\det\left(A_n-wB_n/\sqrt{n}\right)}
{\det\left(A_n\right)},
\qquad
w\in\mathbb{C}.
\label{eq:normalized-determinant}
\end{equation}
Thus, \(Q_n\) is a random polynomial satisfying \(Q_n\left(0\right)=1\).

Let \(G_n^{\mathrm{Gin}}\) be an \(n\times n\) complex Ginibre matrix with variance-one entries. We denote by
\[
0<X_1<X_2<\cdots
\]
the ordered points of the complex hard-edge process in the normalization determined, for every fixed \(K\geq 1\), by
\[
\left(
n s_1\left(G_n^{\mathrm{Gin}}\right)^2,
\ldots,
n s_K\left(G_n^{\mathrm{Gin}}\right)^2
\right)
\overset{\mathrm{d}}{\rightarrow}
\left(X_1,\ldots,X_K\right);
\]
see \citet[Theorem~6.2]{TaoVu2010}. Set
\[
D_j:=X_j^{-1/2},
\qquad
j\geq 1.
\]
Let \(G_\infty=\left(G_{ab}\right)_{a,b\geq 1}\) be an infinite array of independent standard complex Gaussian random variables, independent of \(\left(X_j\right)_{j\geq 1}\), and let
\[
G_K:=\left(G_{ab}\right)_{1\leq a,b\leq K}.
\]
For \(K\geq 1\), define the random polynomial
\begin{equation}
Q^{[K]}\left(w\right)
:=
\det\left(
I_K
-
w\operatorname{diag}\left(D_1,\ldots,D_K\right)G_K
\right).
\label{eq:limiting-determinant-truncation}
\end{equation}

\begin{theorem}[Scalar hard-edge determinant field]
\label{thm:determinant-field}
Under Assumption~\ref{ass:matrix-atoms}, there exists a random entire function
\(Q\in\operatorname{Hol}\left(\mathbb{C}\right)\), defined on the same probability space as
\(\left(X_j\right)_{j\geq 1}\) and \(G_\infty\), such that
\begin{equation}
Q^{[K]}
\overset{\mathbb{P}}{\rightarrow}
Q
\quad\text{in }\operatorname{Hol}\left(\mathbb{C}\right)
\quad\text{as }K\to\infty,
\qquad
Q_n
\overset{\mathrm{d}}{\rightarrow}
Q
\quad\text{in }\operatorname{Hol}\left(\mathbb{C}\right)
\quad\text{as }n\to\infty.
\label{eq:determinant-field-convergence}
\end{equation}
Moreover,
\(
Q\left(0\right)=1
\text{  almost surely}.
\)
We call \(Q\) the scalar hard-edge determinant field. Here, field refers to
a random holomorphic function of the perturbation parameter
\end{theorem}

\begin{proof}
Recall that \(M_n=A_nB_n^{-1}\). Since \(B_n\) is almost surely invertible,
\[
A_n-\frac{w}{\sqrt{n}}B_n
=
\left(
M_n-\frac{w}{\sqrt{n}}I_n
\right)B_n.
\]
Consequently, \(Q_n\left(w\right)=0\) if and only if
\(w/\sqrt{n}\in\operatorname{spec}\left(M_n\right)\). Multiplication by the nonzero constant
\(\det\left(B_n\right)/\det\left(A_n\right)\), together with the linear change of variable
\(w\mapsto w/\sqrt{n}\), preserves multiplicities. Hence
\begin{equation}
\mathcal{Z}\left(Q_n\right)
=
\sum_{\lambda\in\operatorname{spec}\left(M_n\right)}
\delta_{\sqrt{n}\lambda},
\label{eq:determinant-zeros}
\end{equation}
where the sum is over the spectrum with algebraic multiplicity.

Choose an ordered measurable singular-value decomposition
\[
A_n
=
U_n\Sigma_nV_n^*,
\qquad
\Sigma_n
=
\operatorname{diag}\left(
s_1\left(A_n\right),
\ldots,
s_n\left(A_n\right)
\right).
\]
Define
\[
D_{n,j}
:=
\frac{1}{\sqrt{n}\,s_j\left(A_n\right)},
\qquad
D_n
:=
\operatorname{diag}\left(D_{n,1},\ldots,D_{n,n}\right),
\qquad
C_n
:=
U_n^*B_nV_n.
\]
Then
\[
A_n-\frac{w}{\sqrt{n}}B_n
=
U_n
\left(
\Sigma_n-\frac{w}{\sqrt{n}}C_n
\right)
V_n^*.
\]
Dividing its determinant by
\(\det\left(A_n\right)=\det\left(U_n\Sigma_nV_n^*\right)\) gives
\begin{equation}
Q_n\left(w\right)
=
\det\left(I_n-wD_nC_n\right).
\label{eq:singular-value-reduction}
\end{equation}

Absolute continuity of the entries implies that the singular values of \(A_n\) are distinct almost surely. The remaining ambiguity in the ordered singular-value decomposition consists of replacing
\(
U_n\longmapsto U_n\Lambda_n
\)
and
\(
V_n\longmapsto V_n\Lambda_n,
\)
where \(\Lambda_n\) is diagonal and unitary. Under this replacement,
\(C_n\) becomes \(\Lambda_n^*C_n\Lambda_n\). Since \(D_n\) commutes with
\(\Lambda_n\), the matrix in \eqref{eq:singular-value-reduction} is conjugated by
\(\Lambda_n\), and its determinant is unchanged.

Write \(\left[n\right]:=\left\{1,\ldots,n\right\}\). Expanding the determinant in \eqref{eq:singular-value-reduction} over principal subsets gives
\begin{equation}
Q_n\left(w\right)
=
\sum_{S\subseteq\left[n\right]}
\left(-w\right)^{\left|S\right|}
\left(\prod_{j\in S}D_{n,j}\right)
\det\left(\left(C_n\right)_{S,S}\right).
\label{eq:principal-minor-expansion}
\end{equation}
We claim that, conditional on \(A_n\), these principal minors are orthogonal in \(L^2\):
\begin{equation}
\E\left[
\det\left(\left(C_n\right)_{S,S}\right)
\overline{\det\left(\left(C_n\right)_{T,T}\right)}
\,\middle|\,
A_n
\right]
=
\left|S\right|!\,
\mathbf{1}_{\left\{S=T\right\}}
\qquad
\text{for all }S,T\subseteq\left[n\right].
\label{eq:principal-minor-orthogonality}
\end{equation}

Conditional on \(A_n\), the matrices \(U_n\) and \(V_n\) are deterministic, while \(B_n\) remains independent of them. Let \(k=\left|S\right|\). Applying the Cauchy--Binet formula first to
\[
\left(C_n\right)_{S,S}
=
\left(U_n^*\right)_{S,\left[n\right]}
B_n
\left(V_n\right)_{\left[n\right],S}
\]
and then to each resulting product gives
\[
\det\left(\left(C_n\right)_{S,S}\right)
=
\sum_{\substack{I,J\subseteq\left[n\right]\\ \left|I\right|=\left|J\right|=k}}
\overline{\det\left(\left(U_n\right)_{I,S}\right)}
\det\left(\left(B_n\right)_{I,J}\right)
\det\left(\left(V_n\right)_{J,S}\right).
\]
It therefore remains to compute the mixed second moments of the minors of \(B_n\).

Let \(I=\left\{i_1<\cdots<i_k\right\}\), \(J=\left\{j_1<\cdots<j_k\right\}\), \(I'=\left\{i'_1<\cdots<i'_\ell\right\}\), and \(J'=\left\{j'_1<\cdots<j'_\ell\right\}\). Expanding both determinants into permutations yields
\[
\begin{aligned}
&\E\left[
\det\left(\left(B_n\right)_{I,J}\right)
\overline{\det\left(\left(B_n\right)_{I',J'}\right)}
\right]  =
\sum_{\sigma\in\mathfrak{S}_k}
\sum_{\tau\in\mathfrak{S}_\ell}
\operatorname{sgn}\left(\sigma\right)
\operatorname{sgn}\left(\tau\right)
\E\left[
\prod_{a=1}^{k}
\left(B_n\right)_{i_a,j_{\sigma\left(a\right)}}
\prod_{b=1}^{\ell}
\overline{\left(B_n\right)_{i'_b,j'_{\tau\left(b\right)}}}
\right].
\end{aligned}
\]
Each determinant monomial uses every matrix position at most once. Suppose that
\[
\left\{
\left(i_a,j_{\sigma\left(a\right)}\right):
1\leq a\leq k
\right\}
\neq
\left\{
\left(i'_b,j'_{\tau\left(b\right)}\right):
1\leq b\leq \ell
\right\}.
\]
A position belonging to the symmetric difference then appears exactly once in the product inside the expectation. Independence of the entries allows its expectation to be factored from the remaining terms, and this factor is either
\[
\E\left[\left(B_n\right)_{ij}\right]=0
\qquad\text{or}\qquad
\E\left[\overline{\left(B_n\right)_{ij}}\right]=0.
\]
Hence the corresponding summand vanishes.

A summand can therefore survive only when the two sets of matrix positions coincide. Projecting these sets onto their first and second coordinates gives
\(
k=\ell,
I=I',
J=J'.
\)
With the indices written in increasing order, equality of the position sets then forces
\(\sigma=\tau\). For every such surviving pair,
\[
\E\left[
\prod_{a=1}^{k}
\left(B_n\right)_{i_a,j_{\sigma\left(a\right)}}
\overline{\left(B_n\right)_{i_a,j_{\sigma\left(a\right)}}}
\right]
=
\prod_{a=1}^{k}
\E\left[
\left|\left(B_n\right)_{i_a,j_{\sigma\left(a\right)}}\right|^2
\right]
=
1.
\]
There are \(k!\) possible permutations, and
\(\operatorname{sgn}\left(\sigma\right)^2=1\). Consequently,
\[
\E\left[
\det\left(\left(B_n\right)_{I,J}\right)
\overline{\det\left(\left(B_n\right)_{I',J'}\right)}
\right]
=
k!\,
\mathbf{1}_{\left\{
k=\ell,\,
I=I',\,
J=J'
\right\}}.
\]

We now substitute this identity into the two expansions. If \(k=\left|S\right|\) and \(\ell=\left|T\right|\), then
\[
\begin{aligned}
&\E\left[
\det\left(\left(C_n\right)_{S,S}\right)
\overline{\det\left(\left(C_n\right)_{T,T}\right)}
\,\middle|\,
A_n
\right] \\
&\quad =
k!\,\mathbf{1}_{\left\{k=\ell\right\}}
\left(
\sum_{\substack{I\subseteq\left[n\right] \left|I\right|=k}}
\overline{\det\left(\left(U_n\right)_{I,S}\right)}
\det\left(\left(U_n\right)_{I,T}\right)
\right)
\left(
\sum_{\substack{J\subseteq\left[n\right]\\ \left|J\right|=k}}
\det\left(\left(V_n\right)_{J,S}\right)
\overline{\det\left(\left(V_n\right)_{J,T}\right)}
\right) \\
&\quad =
k!\,\mathbf{1}_{\left\{k=\ell\right\}}
\det\left(\left(U_n^*U_n\right)_{S,T}\right)
\det\left(\left(V_n^*V_n\right)_{T,S}\right).
\end{aligned}
\]
Since \(U_n\) and \(V_n\) are unitary, each determinant on the last line equals one when \(S=T\) and zero otherwise. This proves \eqref{eq:principal-minor-orthogonality}.

Extend the inverse singular-value coordinates by zeros and write
\[
\boldsymbol{D}_n
:=
\left(
D_{n,1},\ldots,D_{n,n},0,0,\ldots
\right)
\in\ell^2,
\qquad
\boldsymbol{D}
:=
\left(D_1,D_2,\ldots\right).
\]
We prove that
\[
\boldsymbol{D}_n
\overset{\mathrm{d}}{\rightarrow}
\boldsymbol{D}
\qquad\text{in }\ell^2.
\]

For \(1\leq j\leq n\), set
\[
\Lambda_{n,j}
:=
n s_j\left(A_n\right)^2.
\]
If \(G_n^{\mathrm{Gin}}\) is a raw \(n\times n\) complex Ginibre matrix, define similarly
\[
\Lambda_{n,j}^{\mathrm{Gin}}
:=
n s_j\left(G_n^{\mathrm{Gin}}\right)^2.
\]
Thus,
\[
D_{n,j}^2=\frac{1}{\Lambda_{n,j}}.
\]

The growing-window form of \citet[Theorem~6.2]{TaoVu2010} supplies a constant
\(c_{\mathrm{TV}}>0\), which may be decreased below without further mention, such that the following comparison holds. Let
\[
m_n:=\left\lfloor n^\alpha\right\rfloor,
\qquad
0<\alpha<c_{\mathrm{TV}},
\qquad
\delta_n:=n^{-c_{\mathrm{TV}}},
\]
and define
\[
\boldsymbol{\Lambda}_n^{\left(m_n\right)}
:=
\left(
\Lambda_{n,1},\ldots,\Lambda_{n,m_n}
\right),
\qquad
\boldsymbol{\Lambda}_{n,\mathrm{Gin}}^{\left(m_n\right)}
:=
\left(
\Lambda_{n,1}^{\mathrm{Gin}},\ldots,
\Lambda_{n,m_n}^{\mathrm{Gin}}
\right).
\]
For every measurable \(\Omega\subset\mathbb{R}^{m_n}\),
\[
\begin{aligned}
&\mathbb{P}\left(
\boldsymbol{\Lambda}_{n,\mathrm{Gin}}^{\left(m_n\right)}
\in
\Omega\setminus\partial_{\delta_n}\Omega
\right)
-\delta_n \leq
\mathbb{P}\left(
\boldsymbol{\Lambda}_n^{\left(m_n\right)}
\in\Omega
\right) \leq
\mathbb{P}\left(
\boldsymbol{\Lambda}_{n,\mathrm{Gin}}^{\left(m_n\right)}
\in
\Omega\cup\partial_{\delta_n}\Omega
\right)
+\delta_n,
\end{aligned}
\]
where \(\partial_{\delta}\Omega\) denotes the set of points at
\(\ell^\infty\)-distance at most \(\delta\) from the topological boundary of
\(\Omega\).

We briefly justify the coupling consequence that will be used below. For a closed set
\(F\subset\mathbb{R}^{m_n}\), the upper comparison gives
\[
\mathbb{P}\left(
\boldsymbol{\Lambda}_n^{\left(m_n\right)}\in F
\right)
\leq
\mathbb{P}\left(
\boldsymbol{\Lambda}_{n,\mathrm{Gin}}^{\left(m_n\right)}
\in F^{\delta_n}
\right)
+\delta_n,
\]
where
\(
F^r
:=
\left\{
\boldsymbol{x}:
\operatorname{dist}_{\infty}\left(\boldsymbol{x},F\right)<r
\right\}.
\)
Applying the lower comparison with \(\Omega=F^{2\delta_n}\) gives the reverse Prokhorov inequality, because
\[
F
\subset
F^{2\delta_n}
\setminus
\partial_{\delta_n}\left(F^{2\delta_n}\right).
\]
The Prokhorov distance between the two vectors is therefore at most
\(2\delta_n\). By Strassen's coupling theorem, they may be realized on a common probability space so that
\[
\mathbb{P}\left(
\max_{1\leq j\leq m_n}
\left|
\Lambda_{n,j}
-
\Lambda_{n,j}^{\mathrm{Gin}}
\right|
>
2\delta_n
\right)
\leq
2\delta_n.
\]
This use of \citet[Theorem~6.2]{TaoVu2010} concerns the fixed atom law
\(\Xi_A\); no uniformity over \(n\)-dependent atom laws is invoked.

Choose \(\beta>0\) such that
\[
\alpha+2\beta<c_{\mathrm{TV}},
\qquad
a_n:=n^{-\beta}.
\]
The exact complex-Gaussian smallest-singular-value formula in
\citet[Theorem~1.1]{TaoVu2010} gives
\[
\Lambda_{n,1}^{\mathrm{Gin}}
\sim
\operatorname{Exp}\left(1\right),
\qquad
\mathbb{P}\left(
\Lambda_{n,1}^{\mathrm{Gin}}<a_n
\right)
\leq a_n.
\]
Consider the event
\[
\mathcal{E}_n
:=
\left\{
\max_{1\leq j\leq m_n}
\left|
\Lambda_{n,j}
-
\Lambda_{n,j}^{\mathrm{Gin}}
\right|
\leq
2\delta_n
\right\}
\cap
\left\{
\Lambda_{n,1}^{\mathrm{Gin}}\geq a_n
\right\}.
\]
Then
\(
\mathbb{P}\left(\mathcal{E}_n^c\right)
\leq
2\delta_n+a_n
\rightarrow 0.
\)
Since the Gaussian singular values are increasingly ordered,
\(\Lambda_{n,j}^{\mathrm{Gin}}\geq a_n\) for every \(j\leq m_n\) on
\(\mathcal{E}_n\). Moreover, \(2\delta_n\leq a_n/2\) for all sufficiently large \(n\). Hence,
\[
\Lambda_{n,j}
\geq
\Lambda_{n,j}^{\mathrm{Gin}}-2\delta_n
\geq
\frac{1}{2}\Lambda_{n,j}^{\mathrm{Gin}},
\qquad
1\leq j\leq m_n.
\]
For these indices,
\[
\begin{aligned}
\frac{1}{\Lambda_{n,j}}
-
\frac{1}{\Lambda_{n,j}^{\mathrm{Gin}}}
=
\frac{
\Lambda_{n,j}^{\mathrm{Gin}}-\Lambda_{n,j}
}{
\Lambda_{n,j}\Lambda_{n,j}^{\mathrm{Gin}}
}
\leq
\frac{4\delta_n}{a_n^2}.
\end{aligned}
\]
Consequently, for every fixed \(K<m_n\),
\[
\sum_{j=K+1}^{m_n}D_{n,j}^2
\leq
\sum_{j=K+1}^{m_n}
\frac{1}{\Lambda_{n,j}^{\mathrm{Gin}}}
+
\frac{4m_n\delta_n}{a_n^2}
\text{ on }\mathcal{E}_n.
\]
Our choice of \(\alpha\) and \(\beta\) ensures that
\[
\frac{m_n\delta_n}{a_n^2}
\leq
n^{\alpha-c_{\mathrm{TV}}+2\beta}
\rightarrow 0.
\]

We next estimate the Gaussian sum uniformly in \(n\). Delete \(K\) columns from
\(G_n^{\mathrm{Gin}}\), and denote the resulting \(n\times\left(n-K\right)\) matrix by
\(Z_{n,K}\). Let
\[
0<\Theta_{n,1}\leq\cdots\leq\Theta_{n,n-K}
\]
be the eigenvalues of \(Z_{n,K}^*Z_{n,K}\). Since this Gram matrix is a principal submatrix of
\(\left(G_n^{\mathrm{Gin}}\right)^*G_n^{\mathrm{Gin}}\), Cauchy interlacing gives
\[
\Theta_{n,j}
\leq
s_{j+K}\left(G_n^{\mathrm{Gin}}\right)^2,
\qquad
1\leq j\leq n-K.
\]
Therefore,
\[
\sum_{j=K+1}^{n}
\frac{1}{\Lambda_{n,j}^{\mathrm{Gin}}}
=
\sum_{j=1}^{n-K}
\frac{1}{
n s_{j+K}\left(G_n^{\mathrm{Gin}}\right)^2
}
\leq
\frac{1}{n}
\operatorname{tr}\left(
\left(Z_{n,K}^*Z_{n,K}\right)^{-1}
\right).
\]

To compute the expectation on the right, write
\[
Z_{n,K}
=
\left[
\boldsymbol{Z}_1
\ \cdots\
\boldsymbol{Z}_{n-K}
\right].
\]
For \(1\leq q\leq n-K\), let \(P_{-q}\) be the orthogonal projection onto the span of all columns except
\(\boldsymbol{Z}_q\). The Schur-complement formula gives
\[
\left[
\left(Z_{n,K}^*Z_{n,K}\right)^{-1}
\right]_{qq}
=
\frac{1}{
\boldsymbol{Z}_q^*
\left(I_n-P_{-q}\right)
\boldsymbol{Z}_q
}.
\]
Conditional on the remaining columns, \(I_n-P_{-q}\) is a projection of rank
\(K+1\), and \(\boldsymbol{Z}_q\) is an independent standard complex Gaussian vector. Hence,
\[
\boldsymbol{Z}_q^*
\left(I_n-P_{-q}\right)
\boldsymbol{Z}_q
\sim
\operatorname{Gamma}\left(K+1,1\right).
\]
If \(Y\sim\operatorname{Gamma}\left(K+1,1\right)\), then
\[
\E\left[\frac{1}{Y}\right]
=
\frac{1}{K}.
\]
Summing over the \(n-K\) diagonal entries yields
\[
\E\left[
\sum_{j=K+1}^{n}
\frac{1}{\Lambda_{n,j}^{\mathrm{Gin}}}
\right]
\leq
\frac{1}{n}
\E\left[
\operatorname{tr}\left(
\left(Z_{n,K}^*Z_{n,K}\right)^{-1}
\right)
\right]
=
\frac{n-K}{nK}
\leq
\frac{1}{K}.
\]
Markov's inequality and the preceding coupling now imply that, for every
\(\varepsilon>0\),
\[
\limsup_{n\to\infty}
\mathbb{P}\left(
\sum_{j=K+1}^{m_n}D_{n,j}^2>\varepsilon
\right)
\leq
\frac{2}{K\varepsilon}.
\]

It remains to control the indices beyond the growing hard-edge window. Set
\[
L_{n,j}
:=
s_j\left(\frac{A_n}{\sqrt{n}}\right)^2
=
\frac{s_j\left(A_n\right)^2}{n}.
\]
These are the increasingly ordered eigenvalues of
\(
\frac{1}{n}A_nA_n^*.
\)
For the constant variance profile \(1/n\), the deterministic density in
\citet[Theorem~2.7]{AltErdosKruger2017} is the Marchenko--Pastur density
\[
\varrho_{\mathrm{MP}}\left(x\right)
=
\frac{1}{2\pi}
\sqrt{\frac{4-x}{x}}\,
\mathbf{1}_{\left(0,4\right)}\left(x\right).
\]
The matrix \(A_n/\sqrt{n}\) satisfies the assumptions of that theorem: its variance profile is constant and fully indecomposable, it is square, and the fixed moment bounds for \(\Xi_A\) give their moment condition.

For \(1\leq j\leq n/2\), define the deterministic quantiles by
\[
\int_0^{\gamma_{n,j}}
\varrho_{\mathrm{MP}}\left(x\right)
\,\mathrm{d}x
=
\frac{j}{n}.
\]
Since
\(\varrho_{\mathrm{MP}}\left(x\right)\asymp x^{-1/2}\) uniformly between zero and the median of the distribution, there exist constants \(c,C>0\) such that
\[
c\left(\frac{j}{n}\right)^2
\leq
\gamma_{n,j}
\leq
C\left(\frac{j}{n}\right)^2,
\qquad
1\leq j\leq\frac{n}{2}.
\]
Moreover,
\(\varrho_{\mathrm{MP}}\left(\gamma_{n,j}\right)\) is bounded below uniformly over this range. Thus,
\citet[Theorem~2.7\textup{(iii)} and equation~\textup{(2.11)}]{AltErdosKruger2017} imply that, for every \(\eta>0\) and \(D>0\),
\[
\mathbb{P}\left(
\exists\,1\leq j\leq\frac{n}{2}:
\left|
L_{n,j}-\gamma_{n,j}
\right|
>
\frac{n^\eta}{n}
\left(
\sqrt{\gamma_{n,j}}+\frac{1}{n}
\right)
\right)
\leq
C_{\eta,D}n^{-D}.
\]
The quantile bounds show that the error on this event is at most
\(
C n^{-2+\eta}\left(j+1\right).
\)
Choose \(0<\eta<\alpha\). Uniformly for
\(m_n\leq j\leq n/2\),
\[
\frac{
n^{-2+\eta}\left(j+1\right)
}{
j^2/n^2
}
\leq
C\frac{n^\eta}{j}
\leq
C n^{\eta-\alpha}
\rightarrow 0.
\]
It follows that, with overwhelming probability,
\[
L_{n,j}
\geq
c\frac{j^2}{n^2},
\qquad
m_n\leq j\leq\frac{n}{2}.
\]
Equivalently,
\[
s_j\left(A_n\right)
\geq
c\frac{j}{\sqrt{n}}
\qquad
\text{and}
\qquad
D_{n,j}^2
=
\frac{1}{n^2L_{n,j}}
\leq
\frac{C}{j^2}
\]
throughout this range. At \(j=\lfloor n/2\rfloor\), the same estimate gives
\(L_{n,j}\geq c\). By monotonicity, \(L_{n,j}\geq c\) for every \(j>n/2\), and hence
\(D_{n,j}^2\leq Cn^{-2}\) there. We conclude that, with overwhelming probability,
\[
\sum_{j>m_n}D_{n,j}^2
\leq
C\sum_{j>m_n}^{\lfloor n/2\rfloor}\frac{1}{j^2}
+
\frac{C}{n}
\leq
\frac{C}{m_n}.
\]

Combining the growing-window comparison with the rigidity estimate gives, for every
\(\varepsilon>0\),
\begin{equation}
\lim_{K\to\infty}
\limsup_{n\to\infty}
\mathbb{P}\left(
\sum_{j>K}D_{n,j}^2>\varepsilon
\right)
=
0.
\label{eq:inverse-hard-edge-tail}
\end{equation}

We also verify that the limiting profile belongs to \(\ell^2\). For fixed
\(1\leq K<L\), Gaussian hard-edge convergence and the Portmanteau theorem applied to the nonnegative lower-semicontinuous inverse-sum function give
\[
\begin{aligned}
\E\left[
\sum_{j=K+1}^{L}D_j^2
\right]
\leq
\liminf_{n\to\infty}
\E\left[
\sum_{j=K+1}^{L}
\frac{1}{\Lambda_{n,j}^{\mathrm{Gin}}}
\right]
\leq
\frac{1}{K}.
\end{aligned}
\]
Letting \(L\to\infty\) and using monotone convergence,
\[
\E\left[
\sum_{j>K}D_j^2
\right]
\leq
\frac{1}{K}.
\]
In particular,
\(\sum_{j\geq 2}D_j^2<\infty\) almost surely. Since \(X_1>0\) almost surely, \(D_1<\infty\) almost surely as well, and therefore
\(
\boldsymbol{D}\in\ell^2
\text{ almost surely}.
\)

For every fixed \(K\), the fixed-dimensional part of
\citet[Theorem~6.2]{TaoVu2010}, together with the Gaussian hard-edge limit, gives
\[
\left(
D_{n,1},\ldots,D_{n,K}
\right)
\overset{\mathrm{d}}{\rightarrow}
\left(
D_1,\ldots,D_K
\right).
\]
Let \(\Pi_K\) denote the coordinate projection in \(\ell^2\) onto its first \(K\) coordinates. Then
\[
\Pi_K\boldsymbol{D}_n
\overset{\mathrm{d}}{\rightarrow}
\Pi_K\boldsymbol{D}.
\]

To pass from finite-dimensional convergence to convergence in \(\ell^2\), let
\(\Phi\colon\ell^2\to\mathbb{R}\) satisfy
\[
\left|\Phi\right|\leq 1,
\qquad
\left|
\Phi\left(\boldsymbol{x}\right)
-
\Phi\left(\boldsymbol{y}\right)
\right|
\leq
\left\|
\boldsymbol{x}-\boldsymbol{y}
\right\|_{\ell^2}.
\]
For every \(\zeta>0\),
\[
\left|
\E\left[
\Phi\left(\boldsymbol{D}_n\right)
\right]
-
\E\left[
\Phi\left(\Pi_K\boldsymbol{D}_n\right)
\right]
\right|
\leq
\zeta
+
2\mathbb{P}\left(
\left\|
\left(I-\Pi_K\right)\boldsymbol{D}_n
\right\|_{\ell^2}
>
\zeta
\right),
\]
and the analogous inequality holds for \(\boldsymbol{D}\). Hence,
\[
\begin{aligned}
\limsup_{n\to\infty}
\left|
\E\left[
\Phi\left(\boldsymbol{D}_n\right)
\right]
-
\E\left[
\Phi\left(\boldsymbol{D}\right)
\right]
\right|
&\leq
2\zeta\\
&\quad+
2\limsup_{n\to\infty}
\mathbb{P}\left(
\left\|
\left(I-\Pi_K\right)\boldsymbol{D}_n
\right\|_{\ell^2}
>
\zeta
\right)\\
&\quad+
2\mathbb{P}\left(
\left\|
\left(I-\Pi_K\right)\boldsymbol{D}
\right\|_{\ell^2}
>
\zeta
\right).
\end{aligned}
\]
Equation~\eqref{eq:inverse-hard-edge-tail} and the almost-sure membership
\(\boldsymbol{D}\in\ell^2\) make the two probabilities vanish as \(K\to\infty\). Letting first \(K\to\infty\) and then \(\zeta\downarrow 0\) proves
\(
\boldsymbol{D}_n
\overset{\mathrm{d}}{\rightarrow}
\boldsymbol{D}
\text{ in }\ell^2.
\)

Fix \(K\geq 1\). For \(1\leq a,b\leq K\), the definition
\(C_n=U_n^*B_nV_n\) gives
\[
\left(C_n\right)_{ab}
=
\sum_{i,j=1}^n
\overline{\left(U_n\right)_{ia}}
\left(B_n\right)_{ij}
\left(V_n\right)_{jb}.
\]
Conditional on \(A_n\), the matrices \(U_n\) and \(V_n\) are deterministic, while the entries of \(B_n\) remain independent copies of \(\Xi_B\). Hence,
\[
\E\left[\left(C_n\right)_{ab}\,\middle|\,A_n\right]=0.
\]
For \(1\leq a,b,c,d\leq K\), independence and unit variance give
\[
\begin{aligned}
\E\left[
\left(C_n\right)_{ab}
\overline{\left(C_n\right)_{cd}}
\,\middle|\,
A_n
\right]
&=
\sum_{i,j=1}^n
\overline{\left(U_n\right)_{ia}}
\left(U_n\right)_{ic}
\left(V_n\right)_{jb}
\overline{\left(V_n\right)_{jd}} \\
&=
\left(U_n^*U_n\right)_{ac}
\left(V_n^*V_n\right)_{db}
=
\delta_{ac}\delta_{bd}.
\end{aligned}
\]
The conditional pseudocovariance is
\[
\begin{aligned}
\E\left[
\left(C_n\right)_{ab}
\left(C_n\right)_{cd}
\,\middle|\,
A_n
\right]
&=
\E\left[\Xi_B^2\right]
\left(
\sum_{i=1}^n
\overline{\left(U_n\right)_{ia}}
\overline{\left(U_n\right)_{ic}}
\right)
\left(
\sum_{j=1}^n
\left(V_n\right)_{jb}
\left(V_n\right)_{jd}
\right) \\
&=0.
\end{aligned}
\]
Thus, for every \(n\), the conditional covariance and pseudocovariance of the leading block of \(C_n\) agree with those of a \(K\times K\) standard complex Ginibre matrix. It remains to verify that no individual entry of \(B_n\) contributes macroscopically to this block.

Consider the Hermitian linearization
\[
H_n
:=
\begin{pmatrix}
0 & A_n/\sqrt{n} \\
A_n^*/\sqrt{n} & 0
\end{pmatrix}
\]
and its resolvent
\[
R_n\left(z\right)
:=
\left(H_n-zI_{2n}\right)^{-1},
\qquad
z\in\mathbb{C}_+.
\]
In the notation of \citet[Theorem~1.1 and equation~\textup{(1.10a)}]{AjankiErdosKruger2014}, the ambient dimension is \(2n\) and one may take \(M=n\). The variance matrix of \(H_n\) has entries \(1/n\) in the two off-diagonal blocks and zero elsewhere. Its row sums equal one and its spectrum is contained in
\(\left\{-1,0,1\right\}\), so the eigenvalue \(-1\) is precisely the imprimitive direction covered by that theorem. The moment condition in Assumption~\ref{ass:matrix-atoms} is also the normalized moment condition required there.

Fix \(\gamma>0\) and set
\[
\eta_n:=n^{-1+\gamma}.
\]
For \(\left|E\right|\leq 1\), the entrywise local law gives
\[
\max_{1\leq x\leq 2n}
\left|
\left(R_n\left(E+\mathrm{i}\eta_n\right)\right)_{xx}
-
m_{\mathrm{sc}}\left(E+\mathrm{i}\eta_n\right)
\right|
\leq
n^\delta
\left(
n^{-\gamma/2}
+
n^{-\gamma}
\right)
\]
with probability tending to one faster than any inverse power of \(n\), for every fixed \(\delta>0\). Here \(m_{\mathrm{sc}}\) is the Stieltjes transform of the semicircle law, which is uniformly bounded on the indicated domain.

To use this estimate at the random eigenvalues of \(H_n\), take a deterministic \(n^{-4}\)-net of \(\left[-1,1\right]\). A union bound preserves the preceding probability estimate on the net. Moreover, the resolvent identity gives
\[
\left\|
R_n\left(E+\mathrm{i}\eta_n\right)
-
R_n\left(E'+\mathrm{i}\eta_n\right)
\right\|_{\mathrm{op}}
\leq
\frac{\left|E-E'\right|}{\eta_n^2}.
\]
Since \(n^{-4}\eta_n^{-2}=n^{-2-2\gamma}\), the estimate extends from the net to every \(E\in\left[-1,1\right]\). Consequently, for every \(\delta>0\),
\[
\sup_{\left|E\right|\leq 1}
\max_{1\leq x\leq 2n}
\Im\left(R_n\left(E+\mathrm{i}\eta_n\right)\right)_{xx}
\leq n^\delta
\]
with probability tending to one.

Let \(\boldsymbol{u}_{n,a}\) and \(\boldsymbol{v}_{n,a}\) denote the \(a\)-th columns of \(U_n\) and \(V_n\). The two normalized eigenvectors of \(H_n\) associated with the eigenvalues
\[
\lambda_{n,a}^{\pm}
=
\pm\frac{s_a\left(A_n\right)}{\sqrt{n}}
\]
are
\[
\frac{1}{\sqrt{2}}
\begin{pmatrix}
\boldsymbol{u}_{n,a} \\
\pm\boldsymbol{v}_{n,a}
\end{pmatrix}.
\]
The finite-dimensional hard-edge convergence already established implies
\[
\max_{1\leq a\leq K}
\frac{s_a\left(A_n\right)}{\sqrt{n}}
\xrightarrow{\mathbb{P}}0,
\]
so these eigenvalues belong to \(\left[-1,1\right]\) with probability tending to one. If \(\boldsymbol{q}\) is a normalized eigenvector of \(H_n\) with eigenvalue \(\lambda\), the spectral decomposition of the resolvent yields
\[
\Im\left(R_n\left(\lambda+\mathrm{i}\eta_n\right)\right)_{xx}
=
\sum_{\mu\in\operatorname{spec}\left(H_n\right)}
\frac{\eta_n\left|\boldsymbol{q}_{\mu}\left(x\right)\right|^2}
{\left(\mu-\lambda\right)^2+\eta_n^2}
\geq
\frac{\left|\boldsymbol{q}\left(x\right)\right|^2}{\eta_n}.
\]
Therefore, for every \(\varepsilon>0\),
\begin{equation}
\max_{\substack{1\leq i\leq n\\1\leq a\leq K}}
\left|\left(U_n\right)_{ia}\right|
+
\max_{\substack{1\leq j\leq n\\1\leq b\leq K}}
\left|\left(V_n\right)_{jb}\right|
\leq
n^{-1/2+\varepsilon}
\label{eq:fixed-singular-vector-delocalization}
\end{equation}
with probability tending to one. Indeed, choose \(\gamma,\delta>0\) such that
\(\left(\gamma+\delta\right)/2<\varepsilon\) and combine the two preceding displays.

We now apply a conditional central limit theorem. Let
\(T=\left(t_{ab}\right)_{a,b=1}^K\in\mathbb{C}^{K\times K}\) be deterministic and define the real linear functional
\[
L_n\left(T\right)
:=
\Re
\sum_{a,b=1}^K
t_{ab}\left(C_n\right)_{ab}.
\]
Then
\[
L_n\left(T\right)
=
\sum_{i,j=1}^n
\Re\left(
W_{n,ij}\left(T\right)
\left(B_n\right)_{ij}
\right),
\]
where
\[
W_{n,ij}\left(T\right)
:=
\sum_{a,b=1}^K
t_{ab}
\overline{\left(U_n\right)_{ia}}
\left(V_n\right)_{jb}.
\]
Unitarity gives the exact normalization
\[
\begin{aligned}
\sum_{i,j=1}^n
\left|W_{n,ij}\left(T\right)\right|^2
&=
\sum_{a,b,c,d=1}^K
t_{ab}\overline{t_{cd}}
\left(
\sum_{i=1}^n
\overline{\left(U_n\right)_{ia}}
\left(U_n\right)_{ic}
\right)
\left(
\sum_{j=1}^n
\left(V_n\right)_{jb}
\overline{\left(V_n\right)_{jd}}
\right) \\
&=
\sum_{a,b=1}^K
\left|t_{ab}\right|^2
=
\left\|T\right\|_{\mathrm{HS}}^2.
\end{aligned}
\]
Moreover,
\[
\left|W_{n,ij}\left(T\right)\right|
\leq
\left\|T\right\|_{\mathrm{op}}
\left(
\sum_{a=1}^K
\left|\left(U_n\right)_{ia}\right|^2
\right)^{1/2}
\left(
\sum_{b=1}^K
\left|\left(V_n\right)_{jb}\right|^2
\right)^{1/2}.
\]
It follows from \eqref{eq:fixed-singular-vector-delocalization} that
\[
\sup_{\left\|T\right\|_{\mathrm{HS}}=1}
\max_{1\leq i,j\leq n}
\left|W_{n,ij}\left(T\right)\right|
\xrightarrow{\mathbb{P}}0.
\]

The circular second-moment condition gives, conditional on \(A_n\),
\[
\Var\left(
L_n\left(T\right)
\,\middle|\,
A_n
\right)
=
\frac{1}{2}
\sum_{i,j=1}^n
\left|W_{n,ij}\left(T\right)\right|^2
=
\frac{1}{2}
\left\|T\right\|_{\mathrm{HS}}^2.
\]
The third moment of \(\Xi_B\) is finite, and therefore
\[
\begin{aligned}
\sum_{i,j=1}^n
\E\left[
\left|
\Re\left(
W_{n,ij}\left(T\right)
\left(B_n\right)_{ij}
\right)
\right|^3
\,\middle|\,
A_n
\right] 
\leq
\E\left[\left|\Xi_B\right|^3\right]
\max_{1\leq i,j\leq n}
\left|W_{n,ij}\left(T\right)\right|
\sum_{i,j=1}^n
\left|W_{n,ij}\left(T\right)\right|^2
\xrightarrow{\mathbb{P}}0.
\end{aligned}
\]
Thus the conditional Lyapunov condition holds for every real linear functional of the leading block. Conditional Lindeberg--Feller and Cram\'er--Wold give, for every bounded Lipschitz function
\(g\colon\mathbb{C}^{K\times K}\to\mathbb{R}\),
\begin{equation}
\E\left[
g\left(
\left(C_n\right)_{\left[K\right],\left[K\right]}
\right)
\,\middle|\,
A_n
\right]
\xrightarrow{\mathbb{P}}
\E\left[g\left(G_K\right)\right].
\label{eq:conditional-rotated-block-limit}
\end{equation}
The conditional second-moment identity
\[
\E\left[
\left\|
\left(C_n\right)_{\left[K\right],\left[K\right]}
\right\|_{\mathrm{HS}}^2
\,\middle|\,
A_n
\right]
=
K^2
\]
provides the required tightness in this finite-dimensional Cram\'er--Wold argument.

Finally, \(\boldsymbol{D}_n\) is measurable with respect to \(A_n\). If
\(f\colon\ell^2\to\mathbb{R}\) is bounded and continuous and \(g\) is bounded and Lipschitz, then \eqref{eq:conditional-rotated-block-limit} and bounded convergence in probability imply
\[
\begin{aligned}
\E\left[
f\left(\boldsymbol{D}_n\right)
g\left(
\left(C_n\right)_{\left[K\right],\left[K\right]}
\right)
\right] =
\E\left[
f\left(\boldsymbol{D}_n\right)
\E\left[
g\left(
\left(C_n\right)_{\left[K\right],\left[K\right]}
\right)
\,\middle|\,
A_n
\right]
\right] 
\rightarrow
\E\left[f\left(\boldsymbol{D}\right)\right]
\E\left[g\left(G_K\right)\right].
\end{aligned}
\]
Consequently,
\begin{equation}
\left(
\boldsymbol{D}_n,
\left(C_n\right)_{\left[K\right],\left[K\right]}
\right)
\overset{\mathrm{d}}{\rightarrow}
\left(
\boldsymbol{D},
G_K
\right)
\qquad
\text{in }
\ell^2\times\mathbb{C}^{K\times K},
\label{eq:finite-rotated-block-limit}
\end{equation}
where \(G_K\) is independent of \(\boldsymbol{D}\).

For \(1\leq K\leq n\), define
\[
D_n^{[K]}
:=
\operatorname{diag}\left(
D_{n,1},\ldots,D_{n,K},0,\ldots,0
\right)
\]
and
\[
Q_n^{[K]}\left(w\right)
:=
\det\left(I_n-wD_n^{[K]}C_n\right)
=
\det\left(
I_K-w\operatorname{diag}\left(D_{n,1},\ldots,D_{n,K}\right)
\left(C_n\right)_{\left[K\right],\left[K\right]}
\right).
\]
The second identity follows because the last \(n-K\) rows of
\(D_n^{[K]}C_n\) vanish.

The principal-minor expansions of \(Q_n\) and \(Q_n^{[K]}\) are
\[
Q_n\left(w\right)
=
\sum_{S\subseteq\left[n\right]}
\left(-w\right)^{\left|S\right|}
\left(\prod_{j\in S}D_{n,j}\right)
\det\left(\left(C_n\right)_{S,S}\right)
\]
and
\[
Q_n^{[K]}\left(w\right)
=
\sum_{S\subseteq\left[K\right]}
\left(-w\right)^{\left|S\right|}
\left(\prod_{j\in S}D_{n,j}\right)
\det\left(\left(C_n\right)_{S,S}\right).
\]
The conditional orthogonality established above therefore gives the exact identity
\begin{equation}
\E\left[
\left|
Q_n\left(w\right)-Q_n^{[K]}\left(w\right)
\right|^2
\,\middle|\,
A_n
\right]
=
\sum_{\substack{S\subseteq\left[n\right]\\S\nsubseteq\left[K\right]}}
\left|S\right|!\,
\left|w\right|^{2\left|S\right|}
\prod_{j\in S}D_{n,j}^2.
\label{eq:conditional-determinant-tail}
\end{equation}

We now convert this pointwise identity into a compact-uniform estimate. Fix
\(0<R<\rho\), and put \(t:=\rho^2\). If \(f\) is holomorphic on a neighborhood of the closed disc of radius \(\rho\), then Cauchy's formula and the Cauchy--Schwarz inequality give
\[
\sup_{\left|w\right|\leq R}
\left|f\left(w\right)\right|^2
\leq
\left(\frac{\rho}{\rho-R}\right)^2
\frac{1}{2\pi}
\int_0^{2\pi}
\left|f\left(\rho e^{\mathrm{i}\theta}\right)\right|^2
\mathrm{d}\theta.
\]
Applying this inequality conditionally on \(A_n\), followed by
\eqref{eq:conditional-determinant-tail}, yields
\[
\begin{aligned}
&\E\left[
\sup_{\left|w\right|\leq R}
\left|
Q_n\left(w\right)-Q_n^{[K]}\left(w\right)
\right|^2
\,\middle|\,
A_n
\right] \leq
\left(\frac{\rho}{\rho-R}\right)^2
\sum_{\substack{S\subseteq\left[n\right]\\S\nsubseteq\left[K\right]}}
\left|S\right|!\,
t^{\left|S\right|}
\prod_{j\in S}D_{n,j}^2.
\end{aligned}
\]

For \(0\leq J\leq K\), write
\[
\Delta_{n,J}:=\sum_{j>J}D_{n,j}^2
\]
and let
\[
e_a^{\left(n,J\right)}
:=
e_a\left(D_{n,1}^2,\ldots,D_{n,J}^2\right)
\]
be the \(a\)-th elementary symmetric polynomial in the first \(J\) squared coordinates. Split each set \(S\nsubseteq\left[K\right]\) uniquely as
\[
S=A\mathbin{\dot{\cup}}T,
\qquad
A\subseteq\left[J\right],
\qquad
T\subseteq\left\{J+1,\ldots,n\right\},
\]
where \(T\cap\left\{K+1,\ldots,n\right\}\neq\varnothing\). If
\(\left|T\right|=b\geq 1\), then
\[
\sum_{\substack{
T\subseteq\left\{J+1,\ldots,n\right\}\\
\left|T\right|=b\\
T\cap\left\{K+1,\ldots,n\right\}\neq\varnothing
}}
\prod_{j\in T}D_{n,j}^2
\leq
\Delta_{n,K}
\frac{\Delta_{n,J}^{b-1}}{\left(b-1\right)!}.
\]
Indeed, one may first select one element of \(T\) lying above \(K\), then sum over the remaining \(b-1\) elements; possible multiple counting only enlarges the resulting upper bound.

Consequently, whenever \(t\Delta_{n,J}<1\),
\[
\begin{aligned}
\sum_{\substack{S\subseteq\left[n\right]\\S\nsubseteq\left[K\right]}}
\left|S\right|!\,
t^{\left|S\right|}
\prod_{j\in S}D_{n,j}^2 
&\leq
\Delta_{n,K}
\sum_{a=0}^J
e_a^{\left(n,J\right)}t^a
\sum_{b=1}^{\infty}
\frac{\left(a+b\right)!}{\left(b-1\right)!}
t^b\Delta_{n,J}^{b-1} =
t\Delta_{n,K}
\sum_{a=0}^J
\left(a+1\right)!
e_a^{\left(n,J\right)}t^a
\left(1-t\Delta_{n,J}\right)^{-a-2},
\end{aligned}
\]
where the last equality uses
\[
\sum_{b=1}^{\infty}
\frac{\left(a+b\right)!}{\left(b-1\right)!}
x^{b-1}
=
\left(a+1\right)!
\left(1-x\right)^{-a-2},
\qquad
\left|x\right|<1.
\]
We have therefore obtained
\begin{equation}
\begin{aligned}
&\E\left[
\sup_{\left|w\right|\leq R}
\left|
Q_n\left(w\right)-Q_n^{[K]}\left(w\right)
\right|^2
\,\middle|\,
A_n
\right] \leq
\left(\frac{\rho}{\rho-R}\right)^2
t\Delta_{n,K}
\sum_{a=0}^J
\left(a+1\right)!
e_a^{\left(n,J\right)}t^a
\left(1-t\Delta_{n,J}\right)^{-a-2}.
\end{aligned}
\label{eq:compact-determinant-tail-bound}
\end{equation}

To remove the random factors on the right-hand side, fix \(B>0\) and consider the event
\[
\mathcal{E}_{n,J,B}
:=
\left\{
t\Delta_{n,J}\leq\frac{1}{2},
\quad
\max_{1\leq j\leq J}D_{n,j}\leq B
\right\}.
\]
On this event,
\[
e_a^{\left(n,J\right)}
\leq
\binom{J}{a}B^{2a}
\]
and
\[
\left(1-t\Delta_{n,J}\right)^{-a-2}
\leq
2^{a+2}.
\]
Thus, the sum in \eqref{eq:compact-determinant-tail-bound} is bounded by the deterministic finite constant
\[
H\left(J,B,t\right)
:=
\sum_{a=0}^J
\left(a+1\right)!
\binom{J}{a}
B^{2a}t^a2^{a+2}.
\]

For \(\varepsilon,\delta>0\), conditional Markov's inequality now gives
\[
\begin{aligned}
\mathbb{P}\left(
\sup_{\left|w\right|\leq R}
\left|
Q_n\left(w\right)-Q_n^{[K]}\left(w\right)
\right|
>\varepsilon
\right) 
\leq
\mathbb{P}\left(\mathcal{E}_{n,J,B}^c\right)
+
\mathbb{P}\left(\Delta_{n,K}>\delta\right)
+
\frac{1}{\varepsilon^2}
\left(\frac{\rho}{\rho-R}\right)^2
t\delta H\left(J,B,t\right).
\end{aligned}
\]
The inverse hard-edge estimate proved above implies
\[
\lim_{J\to\infty}
\limsup_{n\to\infty}
\mathbb{P}\left(
t\Delta_{n,J}>\frac{1}{2}
\right)
=
0.
\]
For each fixed \(J\), the convergence of the first \(J\) hard-edge coordinates gives tightness of
\(\max_{j\leq J}D_{n,j}\). We may therefore first choose \(J\), then \(B\), so that the first probability on the right-hand side is arbitrarily small in the limit superior over \(n\). After these choices, take \(\delta>0\) sufficiently small and then \(K\geq J\) sufficiently large. The same inverse hard-edge estimate makes the second probability arbitrarily small. We conclude that, for every \(R,\varepsilon>0\),
\begin{equation}
\lim_{K\to\infty}
\limsup_{n\to\infty}
\mathbb{P}\left(
\sup_{\left|w\right|\leq R}
\left|
Q_n\left(w\right)-Q_n^{[K]}\left(w\right)
\right|
>\varepsilon
\right)
=
0.
\label{eq:finite-n-compact-truncation}
\end{equation}

We next construct the limiting entire function from the finite determinants
\(Q^{[K]}\). Recall that
\[
\boldsymbol{D}
=
\left(D_1,D_2,\ldots\right)
\in\ell^2
\qquad\text{almost surely}.
\]
Set
\[
\Delta_K:=\sum_{j>K}D_j^2.
\]
Then \(\Delta_K\downarrow 0\) almost surely. For \(L>K\), the polynomial \(Q^{[K]}\) can be written in dimension \(L\) as
\[
Q^{[K]}\left(w\right)
=
\det\left(
I_L
-
w\operatorname{diag}\left(
D_1,\ldots,D_K,0,\ldots,0
\right)G_L
\right).
\]
Conditional on \(\boldsymbol{D}\), the principal minors of \(G_L\) satisfy the same orthogonality identity used in
\eqref{eq:conditional-determinant-tail}. Repeating the preceding argument therefore gives, for \(J\leq K<L\) and
\(t\Delta_J<1\),
\[
\begin{aligned}
&\E\left[
\sup_{\left|w\right|\leq R}
\left|
Q^{[L]}\left(w\right)-Q^{[K]}\left(w\right)
\right|^2
\,\middle|\,
\boldsymbol{D}
\right] \leq
\left(\frac{\rho}{\rho-R}\right)^2
t\Delta_K
\sum_{a=0}^J
\left(a+1\right)!
e_a\left(D_1^2,\ldots,D_J^2\right)t^a
\left(1-t\Delta_J\right)^{-a-2}.
\end{aligned}
\]
The same localization argument, now using
\(\Delta_J\to 0\) almost surely and the almost-sure finiteness of
\(D_1,\ldots,D_J\), shows that
\begin{equation}
\lim_{K\to\infty}
\sup_{L\geq K}
\mathbb{P}\left(
\sup_{\left|w\right|\leq R}
\left|
Q^{[L]}\left(w\right)-Q^{[K]}\left(w\right)
\right|
>\varepsilon
\right)
=
0
\label{eq:limiting-compact-truncation}
\end{equation}
for every \(R,\varepsilon>0\).

The closed discs of integer radius exhaust \(\mathbb{C}\), so
\eqref{eq:limiting-compact-truncation} makes
\(\left(Q^{[K]}\right)_{K\geq 1}\) Cauchy in probability for the compact-open topology. To see the existence of the limit directly, choose an increasing sequence
\(\left(K_m\right)_{m\geq 1}\) such that
\[
\sup_{L\geq K_m}
\mathbb{P}\left(
\sup_{\left|w\right|\leq m}
\left|
Q^{[L]}\left(w\right)-Q^{[K_m]}\left(w\right)
\right|
>
2^{-m}
\right)
\leq
2^{-m}.
\]
Taking \(L=K_{m+1}\) and applying the Borel--Cantelli lemma shows that
\(\left(Q^{[K_m]}\right)\) is almost surely Cauchy uniformly on every compact subset of
\(\mathbb{C}\). Its limit is therefore an entire function, which we denote by \(Q\). Equation
\eqref{eq:limiting-compact-truncation} then implies that the full sequence satisfies
\[
Q^{[K]}
\overset{\mathbb{P}}{\rightarrow}
Q
\qquad\text{in }\operatorname{Hol}\left(\mathbb{C}\right).
\]
Since \(Q^{[K]}\left(0\right)=1\) for every \(K\), continuity of evaluation at the origin gives
\[
Q\left(0\right)=1
\qquad\text{almost surely}.
\]

For each fixed \(K\), the joint convergence of the hard-edge coordinates and the rotated block in
\eqref{eq:finite-rotated-block-limit}, followed by the continuous mapping theorem, gives
\[
Q_n^{[K]}
\overset{\mathrm{d}}{\rightarrow}
Q^{[K]}
\qquad\text{in }\operatorname{Hol}\left(\mathbb{C}\right).
\]
Combining this fixed-\(K\) convergence with
\eqref{eq:finite-n-compact-truncation} and
\(Q^{[K]}\to Q\) in probability, the converging-together argument yields
\[
Q_n
\overset{\mathrm{d}}{\rightarrow}
Q
\qquad\text{in }\operatorname{Hol}\left(\mathbb{C}\right).
\]
\end{proof}

\subsection{Second-moment universality at the origin}
\label{subsec:microscopic-spectrum}

Define the origin-scale eigenvalue point process
\[
\mathcal{Z}_n
:=
\sum_{\lambda\in\operatorname{spec}\left(M_n\right)}
\delta_{\sqrt{n}\lambda},
\qquad
M_n=A_nB_n^{-1},
\]
where eigenvalues are counted according to their algebraic multiplicities. Let
\(
0<R_{n,1}\leq\cdots\leq R_{n,n}
\)
be the ordered moduli of the eigenvalues of \(M_n\), and recall that
\(
0<R_1<R_2<\cdots
\)
denote the ordered moduli of the points of
\(\operatorname{Gin}_{\infty}\).

\begin{theorem}[Microscopic second-moment universality]
\label{thm:microscopic-spectrum}
Under Assumption~\ref{ass:matrix-atoms}, the following statements hold.

\begin{enumerate}[label=\textup{(\roman*)},leftmargin=*]
\item The rescaled spectrum converges to the infinite complex Ginibre process:
\[
\mathcal{Z}_n
\overset{\mathrm{d}}{\rightarrow}
\operatorname{Gin}_{\infty}
\qquad
\text{in }\mathcal{N}\left(\mathbb{C}\right).
\]

\item For every fixed \(k\geq 1\),
\[
\left(
\sqrt{n}R_{n,1},
\ldots,
\sqrt{n}R_{n,k}
\right)
\overset{\mathrm{d}}{\rightarrow}
\left(
R_1,\ldots,R_k
\right).
\]

\item For every fixed \(k\geq 1\),
\[
\left(
\frac{R_{n,n}}{\sqrt{n}},
\ldots,
\frac{R_{n,n-k+1}}{\sqrt{n}}
\right)
\overset{\mathrm{d}}{\rightarrow}
\left(
R_1^{-1},\ldots,R_k^{-1}
\right).
\]

\item In particular,
\[
\sqrt{n}\rho_{\min}\left(M_n\right)
\overset{\mathrm{d}}{\rightarrow}
R_1,
\qquad
\frac{\rho_{\max}\left(M_n\right)}{\sqrt{n}}
\overset{\mathrm{d}}{\rightarrow}
R_1^{-1}.
\]

\item If \(D\subset\mathbb{C}\) is a bounded Borel set whose boundary has zero planar Lebesgue measure, then
\[
\mathcal{Z}_n\left(D\right)
\overset{\mathrm{d}}{\rightarrow}
\operatorname{Gin}_{\infty}\left(D\right).
\]
\end{enumerate}

More explicitly, let
\(\left(\Gamma_j\right)_{j\geq 1}\) be independent random variables satisfying
\[
\Gamma_j\sim\operatorname{Gamma}\left(j,1\right).
\]
The Kostlan representation gives the equality in distribution of unordered point configurations
\[
\left\{
\left|Z\right|^2:
Z\in\operatorname{Gin}_{\infty}
\right\}
\overset{\mathrm{d}}{=}
\left\{
\Gamma_j:
j\geq 1
\right\};
\]
see \citet[Theorem~4.7.3 and Section~7.2]{HoughKrishnapurPeresVirag2009}.
Consequently, \(\left(R_1,\ldots,R_k\right)\) are the first \(k\) order statistics of
\[
\left\{
\sqrt{\Gamma_j}:
j\geq 1
\right\}.
\]

For \(a>0\), let
\[
D_a:=\left\{z\in\mathbb{C}:\left|z\right|\leq a\right\}.
\]
Then
\[
\operatorname{Gin}_{\infty}\left(D_a\right)
\overset{\mathrm{d}}{=}
\sum_{j=1}^{\infty}I_j\left(a\right),
\]
where the variables \(\left(I_j\left(a\right)\right)_{j\geq 1}\) are independent and
\[
I_j\left(a\right)
\sim
\operatorname{Bernoulli}\left(p_j\left(a\right)\right),
\qquad
p_j\left(a\right)
=
\Pr\left(\Gamma_j\leq a^2\right)
=
1-e^{-a^2}
\sum_{\ell=0}^{j-1}
\frac{a^{2\ell}}{\ell!}.
\]
\end{theorem}

\begin{proof}
By the determinant identity established in
\eqref{eq:determinant-zeros},
\[
\mathcal{Z}_n
=
\mathcal{Z}\left(Q_n\right).
\]
Theorem~\ref{thm:determinant-field} gives
\[
Q_n
\overset{\mathrm{d}}{\rightarrow}
Q
\qquad
\text{in }\operatorname{Hol}\left(\mathbb{C}\right),
\]
and \(Q\left(0\right)=1\) almost surely. In particular, the limiting function is almost surely not identically zero.

We first recall the relevant continuity property. Suppose that
\(f_m,f\in\operatorname{Hol}\left(\mathbb{C}\right)\),
that \(f_m\to f\) uniformly on compact sets, and that \(f\not\equiv 0\).
Fix a compact set \(K\subset\mathbb{C}\). Since the zeros of \(f\) are isolated, one may choose a bounded open set \(U\) containing \(K\) such that \(f\) has no zeros on \(\partial U\). The finitely many zeros of \(f\) in \(\overline{U}\) can be enclosed in pairwise disjoint discs whose boundaries contain no zeros of \(f\). Uniform convergence on these boundaries and Rouché's theorem imply that, for all sufficiently large \(m\), each disc contains the same number of zeros of \(f_m\) and \(f\), counted with multiplicity. Outside these discs, \(f\) is bounded away from zero, so \(f_m\) has no additional zeros there. It follows that
\[
\mathcal{Z}\left(f_m\right)
\rightarrow
\mathcal{Z}\left(f\right)
\]
vaguely. Hence the map
\(
f\longmapsto\mathcal{Z}\left(f\right)
\)
is continuous at every nonzero entire function.

The continuous mapping theorem therefore yields
\begin{equation}
\mathcal{Z}_n
=
\mathcal{Z}\left(Q_n\right)
\overset{\mathrm{d}}{\rightarrow}
\mathcal{Z}\left(Q\right)
\qquad
\text{in }\mathcal{N}\left(\mathbb{C}\right).
\label{eq:zero-process-preliminary-limit}
\end{equation}
It remains to identify the law of the point process on the right-hand side.

For this purpose, specialize Assumption~\ref{ass:matrix-atoms} to independent standard complex Ginibre matrices
\(A_n^{\mathrm{Gin}}\) and \(B_n^{\mathrm{Gin}}\), and define
\[
M_n^{\mathrm{Gin}}
:=
A_n^{\mathrm{Gin}}
\left(B_n^{\mathrm{Gin}}\right)^{-1}.
\]
The determinant-field theorem applies to this choice of atom laws and gives the same limiting random entire function \(Q\). Therefore,
\[
\sum_{\lambda\in\operatorname{spec}\left(M_n^{\mathrm{Gin}}\right)}
\delta_{\sqrt{n}\lambda}
\overset{\mathrm{d}}{\rightarrow}
\mathcal{Z}\left(Q\right).
\]

On the other hand, the eigenvalues of \(M_n^{\mathrm{Gin}}\) form the spherical ensemble. With respect to planar Lebesgue measure, its determinantal kernel is
\[
K_n^{\mathrm{sph}}\left(z,w\right)
=
\frac{n}{\pi}
\frac{
\left(1+z\overline{w}\right)^{n-1}
}{
\left(1+\left|z\right|^2\right)^{\left(n+1\right)/2}
\left(1+\left|w\right|^2\right)^{\left(n+1\right)/2}
};
\]
see \citet[equation~\textup{(2.3)}]{ChafaiGarciaZeladaXu2026}. After the change of scale
\(z\mapsto z/\sqrt{n}\), the corresponding kernel becomes
\[
\begin{aligned}
\frac{1}{n}
K_n^{\mathrm{sph}}
\left(
\frac{z}{\sqrt{n}},
\frac{w}{\sqrt{n}}
\right)
&=
\frac{1}{\pi}
\frac{
\left(
1+\dfrac{z\overline{w}}{n}
\right)^{n-1}
}{
\left(
1+\dfrac{\left|z\right|^2}{n}
\right)^{\left(n+1\right)/2}
\left(
1+\dfrac{\left|w\right|^2}{n}
\right)^{\left(n+1\right)/2}
} \\
&\rightarrow
\frac{1}{\pi}
\exp\left(
z\overline{w}
-\frac{\left|z\right|^2}{2}
-\frac{\left|w\right|^2}{2}
\right)
=
K_{\infty}\left(z,w\right),
\end{aligned}
\]
uniformly for \(\left(z,w\right)\) in compact subsets of
\(\mathbb{C}^2\). By the standard convergence criterion for determinantal point processes, for example
\citet[Proposition~3.10]{ShiraiTakahashi2003},
\[
\sum_{\lambda\in\operatorname{spec}\left(M_n^{\mathrm{Gin}}\right)}
\delta_{\sqrt{n}\lambda}
\overset{\mathrm{d}}{\rightarrow}
\operatorname{Gin}_{\infty}.
\]
Uniqueness of weak limits in
\(\mathcal{N}\left(\mathbb{C}\right)\) now gives
\[
\mathcal{Z}\left(Q\right)
\overset{\mathrm{d}}{=}
\operatorname{Gin}_{\infty}.
\]
Combining this identity with
\eqref{eq:zero-process-preliminary-limit} proves part \textup{(i)}.

We first record the continuity property needed to pass from the point process to its first finitely many ordered moduli. Let
\(\mu_m,\mu\in\mathcal{N}\left(\mathbb{C}\right)\) satisfy
\(\mu_m\to\mu\) vaguely. Suppose that \(\mu\) has no atom at the origin and that its first \(k+1\) ordered moduli
\[
0<r_1<\cdots<r_k<r_{k+1}
\]
are finite and distinct. Fix \(\varepsilon>0\). For each \(1\leq j\leq k\), choose
\[
r_j-\varepsilon<a_j<r_j<b_j<r_j+\varepsilon
\]
so that neither circle \(\left\{z:\left|z\right|=a_j\right\}\) nor
\(\left\{z:\left|z\right|=b_j\right\}\) contains an atom of \(\mu\), and so that
\[
\mu\left(D_{a_j}\right)=j-1,
\qquad
\mu\left(D_{b_j}\right)=j.
\]
Such a choice is possible because \(\mu\) is locally finite and the moduli
\(r_1,\ldots,r_{k+1}\) are distinct. Since \(D_{a_j}\) and \(D_{b_j}\) are relatively compact continuity sets for \(\mu\), vague convergence gives
\[
\mu_m\left(D_{a_j}\right)\rightarrow j-1,
\qquad
\mu_m\left(D_{b_j}\right)\rightarrow j.
\]
The quantities on the left are integer valued. Hence, for all sufficiently large \(m\),
\[
\mu_m\left(D_{a_j}\right)=j-1,
\qquad
\mu_m\left(D_{b_j}\right)=j,
\]
and the \(j\)-th ordered modulus \(r_j\left(\mu_m\right)\) satisfies
\(
a_j<r_j\left(\mu_m\right)\leq b_j.
\)

As \(\varepsilon>0\) was arbitrary,
\[
\left(
r_1\left(\mu_m\right),
\ldots,
r_k\left(\mu_m\right)
\right)
\rightarrow
\left(
r_1,
\ldots,
r_k
\right).
\]

By the Kostlan representation
\citep[Theorem~4.7.3 and Section~7.2]{HoughKrishnapurPeresVirag2009},
the unordered squared moduli of the points of
\(\operatorname{Gin}_{\infty}\) satisfy
\[
\left\{
\left|Z\right|^2:
Z\in\operatorname{Gin}_{\infty}
\right\}
\overset{\mathrm{d}}{=}
\left\{
\Gamma_j:
j\geq 1
\right\},
\]
where the variables \(\Gamma_j\sim\operatorname{Gamma}\left(j,1\right)\) are independent. Since these distributions are continuous and supported on
\(\left(0,\infty\right)\), the Ginibre process has no point at the origin and no two of its points have the same modulus almost surely. It also has infinitely many points almost surely. The continuity property above therefore holds at
\(\operatorname{Gin}_{\infty}\) with probability one.

The ordered moduli of the atoms of \(\mathcal{Z}_n\) are
\[
\sqrt{n}R_{n,1},
\ldots,
\sqrt{n}R_{n,n}.
\]
Part \textup{(i)} and the continuous mapping theorem consequently imply that, for every fixed \(k\geq 1\),
\[
\left(
\sqrt{n}R_{n,1},
\ldots,
\sqrt{n}R_{n,k}
\right)
\overset{\mathrm{d}}{\rightarrow}
\left(
R_1,\ldots,R_k
\right).
\]
This proves part \textup{(ii)}.

To obtain the largest eigenvalue moduli, consider the ratio with the two matrices interchanged,
\[
\widetilde{M}_n:=B_nA_n^{-1},
\]
and let
\(
0<\widetilde{R}_{n,1}\leq\cdots\leq\widetilde{R}_{n,n}
\)
be the ordered moduli of its eigenvalues. Assumption~\ref{ass:matrix-atoms} is unchanged when \(A_n\) and \(B_n\) are interchanged, even when their atom laws differ. Hence the result just proved applies to \(\widetilde{M}_n\):
\[
\left(
\sqrt{n}\widetilde{R}_{n,1},
\ldots,
\sqrt{n}\widetilde{R}_{n,k}
\right)
\overset{\mathrm{d}}{\rightarrow}
\left(
R_1,\ldots,R_k
\right).
\]
By the spectral inversion relation \eqref{eq:spectral-inversion},
\[
\widetilde{R}_{n,j}
=
R_{n,n-j+1}^{-1},
\qquad
1\leq j\leq n.
\]
It follows that
\[
\left(
\frac{\sqrt{n}}{R_{n,n}},
\ldots,
\frac{\sqrt{n}}{R_{n,n-k+1}}
\right)
\overset{\mathrm{d}}{\rightarrow}
\left(
R_1,\ldots,R_k
\right).
\]
Since \(R_j>0\) almost surely for every fixed \(j\), coordinatewise inversion and the continuous mapping theorem yield
\[
\left(
\frac{R_{n,n}}{\sqrt{n}},
\ldots,
\frac{R_{n,n-k+1}}{\sqrt{n}}
\right)
\overset{\mathrm{d}}{\rightarrow}
\left(
R_1^{-1},
\ldots,
R_k^{-1}
\right),
\]
which proves part \textup{(iii)}.

Finally,
\[
R_{n,1}
=
\rho_{\min}\left(M_n\right),
\qquad
R_{n,n}
=
\rho_{\max}\left(M_n\right).
\]
Taking \(k=1\) in the preceding limits gives
\[
\sqrt{n}\rho_{\min}\left(M_n\right)
\overset{\mathrm{d}}{\rightarrow}
R_1,
\qquad
\frac{\rho_{\max}\left(M_n\right)}{\sqrt{n}}
\overset{\mathrm{d}}{\rightarrow}
R_1^{-1}.
\]
The Kostlan representation further identifies
\[
R_1
\overset{\mathrm{d}}{=}
\sqrt{\min_{j\geq 1}\Gamma_j}.
\]
Therefore,
\[
\sqrt{n}\rho_{\min}\left(M_n\right)
\overset{\mathrm{d}}{\rightarrow}
\sqrt{\min_{j\geq 1}\Gamma_j},
\qquad
\frac{\rho_{\max}\left(M_n\right)}{\sqrt{n}}
\overset{\mathrm{d}}{\rightarrow}
\frac{1}{\sqrt{\min_{j\geq 1}\Gamma_j}},
\]
which proves part \textup{(iv)}.

Let \(D\subset\mathbb{C}\) be bounded and Borel, with
\(\operatorname{Leb}_2\left(\partial D\right)=0\). Since the intensity of
\(\operatorname{Gin}_{\infty}\) is \(1/\pi\), we have
\[
\E\left[
\operatorname{Gin}_{\infty}\left(\partial D\right)
\right]
=
\frac{1}{\pi}
\operatorname{Leb}_2\left(\partial D\right)
=
0.
\]
Consequently,
\(
\operatorname{Gin}_{\infty}\left(\partial D\right)=0
\text{ almost surely}.
\)
For locally finite measures, the map
\(
\mu\longmapsto\mu\left(D\right)
\)
is continuous under vague convergence at every \(\mu\) satisfying
\(\mu\left(\partial D\right)=0\). Applying the continuous mapping theorem to part
\textup{(i)} therefore gives
\[
\mathcal{Z}_n\left(D\right)
\overset{\mathrm{d}}{\rightarrow}
\operatorname{Gin}_{\infty}\left(D\right),
\]
which proves part \textup{(v)}.

We now specialize to the centered disc
\[
D_a
=
\left\{
z\in\mathbb{C}:
\left|z\right|\leq a
\right\},
\qquad
a>0,
\]
and write
\[
N_{\infty}\left(a\right)
:=
\operatorname{Gin}_{\infty}\left(D_a\right).
\]
By the Kostlan representation, the unordered squared moduli of the points of
\(\operatorname{Gin}_{\infty}\) have the same law as the independent variables
\(\left(\Gamma_j\right)_{j\geq 1}\). Hence
\[
N_{\infty}\left(a\right)
\overset{\mathrm{d}}{=}
\sum_{j=1}^{\infty}I_j\left(a\right),
\qquad
I_j\left(a\right)
:=
\mathbf{1}_{\left\{\Gamma_j\leq a^2\right\}}.
\]
The variables \(I_j\left(a\right)\) are independent Bernoulli random variables with
success probabilities
\[
\begin{aligned}
p_j\left(a\right)
:=
\Pr\left(\Gamma_j\leq a^2\right) 
=
\int_0^{a^2}
\frac{x^{j-1}e^{-x}}{\left(j-1\right)!}
\,\mathrm{d}x 
=
1-e^{-a^2}
\sum_{\ell=0}^{j-1}
\frac{a^{2\ell}}{\ell!}.
\end{aligned}
\]
Tonelli's theorem and the exponential series give
\[
\begin{aligned}
\sum_{j=1}^{\infty}p_j\left(a\right)
=
\int_0^{a^2}
e^{-x}
\sum_{j=1}^{\infty}
\frac{x^{j-1}}{\left(j-1\right)!}
\,\mathrm{d}x 
=
\int_0^{a^2}1\,\mathrm{d}x
=
a^2.
\end{aligned}
\]
In particular, the Bernoulli sum is finite almost surely.

For \(0<t\leq 1\), independence gives
\[
\E\left[
t^{\sum_{j=1}^{m}I_j\left(a\right)}
\right]
=
\prod_{j=1}^{m}
\left(
1-p_j\left(a\right)+tp_j\left(a\right)
\right).
\]
Letting \(m\to\infty\) and using dominated convergence yields the probability generating function
\[
\E\left[
t^{N_{\infty}\left(a\right)}
\right]
=
\prod_{j=1}^{\infty}
\left(
1-p_j\left(a\right)+tp_j\left(a\right)
\right).
\]
The infinite product converges because
\(\sum_{j\geq 1}p_j\left(a\right)=a^2<\infty\).

The same summability permits termwise calculation of the first two moments. Thus,
\[
\E\left[
N_{\infty}\left(a\right)
\right]
=
\sum_{j=1}^{\infty}p_j\left(a\right)
=
a^2
\]
and
\[
\Var\left(
N_{\infty}\left(a\right)
\right)
=
\sum_{j=1}^{\infty}
p_j\left(a\right)
\left(
1-p_j\left(a\right)
\right).
\]
Indeed, the partial Bernoulli sums converge to
\(N_{\infty}\left(a\right)\) in \(L^2\), since both
\[
\sum_{j>m}p_j\left(a\right)
\quad\text{and}\quad
\sum_{j>m}p_j\left(a\right)
\left(
1-p_j\left(a\right)
\right)
\]
tend to zero.

Finally, independence gives the gap probability
\[
\begin{aligned}
\Pr\left(
N_{\infty}\left(a\right)=0
\right)
=
\prod_{j=1}^{\infty}
\left(
1-p_j\left(a\right)
\right) 
=
\prod_{j=1}^{\infty}
\Pr\left(
\Gamma_j>a^2
\right) 
=
\prod_{j=1}^{\infty}
\left(
e^{-a^2}
\sum_{\ell=0}^{j-1}
\frac{a^{2\ell}}{\ell!}
\right).
\end{aligned}
\]
Therefore,
\[
\mathcal{Z}_n\left(D_a\right)
\overset{\mathrm{d}}{\rightarrow}
\sum_{j=1}^{\infty}I_j\left(a\right),
\]
with the generating function, mean, variance, and gap probability displayed above. 
\end{proof}

\begin{remark}
Part~\textup{(iv)} of Theorem~\ref{thm:microscopic-spectrum} gives an affirmative answer, for atom laws fixed as \(n\) varies, to the question posed by
\citet[Section~1.2.4]{ChafaiGarciaZeladaXu2026}. Their
Theorem~1.3 establishes the inner- and outer-spectral-radius limits under
conditions \textup{(C1)}--\textup{(C3)}, including the third- and fourth-moment
comparison in \textup{(C2)}. Assumption~\ref{ass:matrix-atoms} retains the
bounded-density and finite-moment requirements corresponding to
\textup{(C1)} and \textup{(C3)}, but replaces \textup{(C2)} by the exact
circular second-moment identities in
\eqref{eq:circular-second-moments}. The conclusion obtained here is stronger
at the microscopic origin: the two spectral-radius limits are the \(k=1\)
cases of the ordered-modulus consequences of the point-process limit. The
possible extension to atom laws depending on \(n\) is discussed in
Appendix~\ref{app:triangular-arrays}.
\end{remark}

\subsection{Multivariate extension}
\label{subsec:several-directions}

The scalar field corresponds to one perturbation direction. We now extend
the construction to a fixed finite-dimensional parameter space and show
that its law is determined by the covariance geometry of the perturbation
vectors.

Fix \(r\geq 1\). For each \(n\), let
\(
B_n^{\left(1\right)},\ldots,B_n^{\left(r\right)}
\in\mathbb{C}^{n\times n}
\)
be independent of \(A_n\). Assume that the random vectors
\[
\boldsymbol{\Xi}_{n,ij}
:=
\left(
\left(B_n^{\left(1\right)}\right)_{ij},
\ldots,
\left(B_n^{\left(r\right)}\right)_{ij}
\right)^\top,
\qquad
1\leq i,j\leq n,
\]
are independent and identically distributed copies of a random vector
\(\boldsymbol{\Xi}\in\mathbb{C}^r\) whose law does not depend on \(n\). Suppose that
\[
\E\left[\boldsymbol{\Xi}\right]=\boldsymbol{0},
\qquad
\E\left[
\boldsymbol{\Xi}\boldsymbol{\Xi}^*
\right]
=
\boldsymbol{\Sigma},
\qquad
\E\left[
\boldsymbol{\Xi}\boldsymbol{\Xi}^\top
\right]
=
\boldsymbol{0},
\]
where \(\boldsymbol{\Sigma}\in\mathbb{C}^{r\times r}\) is Hermitian and positive definite and that
\(
\E\left[
\left\|\boldsymbol{\Xi}\right\|_2^p
\right]
<
\infty
\text{ for every }p\geq 1.
\)

For \(\boldsymbol{z}=\left(z_1,\ldots,z_r\right)^\top\in\mathbb{C}^r\), set
\[
B_n\left(\boldsymbol{z}\right)
:=
\sum_{\ell=1}^r
z_\ell B_n^{\left(\ell\right)}
\]
and define
\begin{equation}
Q_n\left(\boldsymbol{z}\right)
:=
\frac{
\det\left(
A_n-n^{-1/2}B_n\left(\boldsymbol{z}\right)
\right)
}{
\det\left(A_n\right)
}.
\label{eq:multivariate-determinant-field}
\end{equation}

Let
\[
\boldsymbol{G}_{ab}
=
\left(
G_{ab}^{\left(1\right)},
\ldots,
G_{ab}^{\left(r\right)}
\right)^\top,
\qquad
a,b\geq 1,
\]
be independent centered complex Gaussian vectors satisfying
\[
\E\left[
\boldsymbol{G}_{ab}\boldsymbol{G}_{cd}^*
\right]
=
\delta_{ac}\delta_{bd}\boldsymbol{\Sigma},
\qquad
\E\left[
\boldsymbol{G}_{ab}\boldsymbol{G}_{cd}^\top
\right]
=
\boldsymbol{0},
\]
and assume that this array is independent of
\(\left(D_j\right)_{j\geq 1}\). For \(1\leq\ell\leq r\), let
\[
G_K^{\left(\ell\right)}
:=
\left(
G_{ab}^{\left(\ell\right)}
\right)_{1\leq a,b\leq K},
\]
and define
\begin{equation}
Q_{\boldsymbol{\Sigma}}^{[K]}
\left(\boldsymbol{z}\right)
:=
\det\left(
I_K
-
\operatorname{diag}\left(D_1,\ldots,D_K\right)
\sum_{\ell=1}^r z_\ell G_K^{\left(\ell\right)}
\right).
\label{eq:multivariate-limiting-truncation}
\end{equation}

\begin{theorem}[Multivariate hard-edge determinant field]
\label{thm:several-directions}
Under the assumptions above, there exists a random element
\[
Q_{\boldsymbol{\Sigma}}
\in
\operatorname{Hol}\left(\mathbb{C}^r\right)
\]
such that
\begin{equation}
Q_{\boldsymbol{\Sigma}}^{[K]}
\overset{\mathbb{P}}{\rightarrow}
Q_{\boldsymbol{\Sigma}}
\quad
\text{in }
\operatorname{Hol}\left(\mathbb{C}^r\right),
\qquad
Q_n
\overset{\mathrm{d}}{\rightarrow}
Q_{\boldsymbol{\Sigma}}
\quad
\text{in }
\operatorname{Hol}\left(\mathbb{C}^r\right).
\label{eq:multivariate-field-convergence}
\end{equation}
Moreover,
\(
Q_{\boldsymbol{\Sigma}}\left(\boldsymbol{0}\right)=1
\text{ almost surely}.
\)

For every nonzero
\(\boldsymbol{v}\in\mathbb{C}^r\), define
\[
\sigma_{\boldsymbol{v}}^2
:=
\boldsymbol{v}^\top
\boldsymbol{\Sigma}
\overline{\boldsymbol{v}}.
\]
Then \(\sigma_{\boldsymbol{v}}>0\), and
\begin{equation}
Q_{\boldsymbol{\Sigma}}
\left(\,\cdot\,\boldsymbol{v}\right)
\overset{\mathrm{d}}{=}
Q\left(
\sigma_{\boldsymbol{v}}\,\cdot
\right)
\qquad
\text{in }
\operatorname{Hol}\left(\mathbb{C}\right),
\label{eq:line-restriction-law}
\end{equation}
where \(Q\) is the scalar limiting entire function from
Theorem~\ref{thm:determinant-field}. Consequently,
\[
\mathcal{Z}\left(
Q_n\left(\,\cdot\,\boldsymbol{v}\right)
\right)
\overset{\mathrm{d}}{\rightarrow}
\sum_{Z\in\operatorname{Gin}_{\infty}}
\delta_{Z/\sigma_{\boldsymbol{v}}}
\qquad
\text{in }
\mathcal{N}\left(\mathbb{C}\right).
\]
\end{theorem}

\begin{proof}
Use the ordered singular-value decomposition
\[
A_n
=
U_n\Sigma_nV_n^*
\]
chosen in the proof of Theorem~\ref{thm:determinant-field}, and retain the notation
\(
D_n
=
\operatorname{diag}\left(
D_{n,1},\ldots,D_{n,n}
\right).
\)

For \(1\leq\ell\leq r\), set
\[
C_n^{\left(\ell\right)}
:=
U_n^*B_n^{\left(\ell\right)}V_n,
\qquad
C_n\left(\boldsymbol{z}\right)
:=
\sum_{\ell=1}^r
z_\ell C_n^{\left(\ell\right)}.
\]
The same determinant reduction used in
\eqref{eq:singular-value-reduction} gives
\[
Q_n\left(\boldsymbol{z}\right)
=
\det\left(
I_n-D_nC_n\left(\boldsymbol{z}\right)
\right).
\]

For \(1\leq a,b\leq K\), define
\[
\boldsymbol{C}_{n,ab}
:=
\left(
\left(C_n^{\left(1\right)}\right)_{ab},
\ldots,
\left(C_n^{\left(r\right)}\right)_{ab}
\right)^\top.
\]
Conditional on \(A_n\), the matrices \(U_n\) and \(V_n\) are deterministic, and
\[
\boldsymbol{C}_{n,ab}
=
\sum_{i,j=1}^n
\overline{\left(U_n\right)_{ia}}
\left(V_n\right)_{jb}
\boldsymbol{\Xi}_{n,ij}.
\]
It follows that
\(
\E\left[
\boldsymbol{C}_{n,ab}
\,\middle|\,
A_n
\right]
=
\boldsymbol{0}.
\)

For \(1\leq a,b,c,d\leq K\), the conditional covariance is
\[
\begin{aligned}
\E\left[
\boldsymbol{C}_{n,ab}
\boldsymbol{C}_{n,cd}^*
\,\middle|\,
A_n
\right]
=
\sum_{i,j=1}^n
\overline{\left(U_n\right)_{ia}}
\left(U_n\right)_{ic}
\left(V_n\right)_{jb}
\overline{\left(V_n\right)_{jd}}
\boldsymbol{\Sigma} 
=
\delta_{ac}\delta_{bd}\boldsymbol{\Sigma}.
\end{aligned}
\]
Similarly,
\[
\begin{aligned}
\E\left[
\boldsymbol{C}_{n,ab}
\boldsymbol{C}_{n,cd}^\top
\,\middle|\,
A_n
\right]
&=
\left(
\sum_{i=1}^n
\overline{\left(U_n\right)_{ia}}
\overline{\left(U_n\right)_{ic}}
\right)
\left(
\sum_{j=1}^n
\left(V_n\right)_{jb}
\left(V_n\right)_{jd}
\right)
\E\left[
\boldsymbol{\Xi}\boldsymbol{\Xi}^\top
\right] 
=
\boldsymbol{0}.
\end{aligned}
\]

The singular-vector delocalization established in the finite-rotated-block part of the proof of
Theorem~\ref{thm:determinant-field} gives, for every fixed \(K\),
\[
\max_{\substack{1\leq i\leq n\\1\leq a\leq K}}
\left|\left(U_n\right)_{ia}\right|
+
\max_{\substack{1\leq j\leq n\\1\leq b\leq K}}
\left|\left(V_n\right)_{jb}\right|
\overset{\mathbb{P}}{\rightarrow}
0.
\]
We use this estimate to verify the conditional Lindeberg condition jointly over all perturbation directions.

Let
\[
T^{\left(\ell\right)}
=
\left(
t_{ab}^{\left(\ell\right)}
\right)_{a,b=1}^K
\in\mathbb{C}^{K\times K},
\qquad
1\leq\ell\leq r,
\]
be deterministic, and consider
\[
L_n
:=
\Re
\sum_{\ell=1}^r
\sum_{a,b=1}^K
\overline{
t_{ab}^{\left(\ell\right)}
}
\left(C_n^{\left(\ell\right)}\right)_{ab}.
\]
Then
\(
L_n
=
\sum_{i,j=1}^n
\Re\left(
\boldsymbol{\beta}_{n,ij}^\top
\boldsymbol{\Xi}_{n,ij}
\right),
\)
where the \(\ell\)-th coordinate of
\(\boldsymbol{\beta}_{n,ij}\in\mathbb{C}^r\) is
\[
\left(
\boldsymbol{\beta}_{n,ij}
\right)_\ell
=
\sum_{a,b=1}^K
\overline{
t_{ab}^{\left(\ell\right)}
}
\overline{\left(U_n\right)_{ia}}
\left(V_n\right)_{jb}.
\]
Unitarity gives
\(
\sum_{i,j=1}^n
\left\|
\boldsymbol{\beta}_{n,ij}
\right\|_2^2
=
\sum_{\ell=1}^r
\left\|
T^{\left(\ell\right)}
\right\|_{\mathrm{HS}}^2.
\)

Moreover,
\[
\begin{aligned}
\left\|
\boldsymbol{\beta}_{n,ij}
\right\|_2
&\leq
\left(
\sum_{\ell=1}^r
\left\|
T^{\left(\ell\right)}
\right\|_{\mathrm{op}}^2
\right)^{1/2}
\left(
\sum_{a=1}^K
\left|\left(U_n\right)_{ia}\right|^2
\right)^{1/2}
\left(
\sum_{b=1}^K
\left|\left(V_n\right)_{jb}\right|^2
\right)^{1/2}.
\end{aligned}
\]
Hence
\(
\max_{1\leq i,j\leq n}
\left\|
\boldsymbol{\beta}_{n,ij}
\right\|_2
\overset{\mathbb{P}}{\rightarrow}
0.
\)

For every \(\varepsilon>0\),
\[
\begin{aligned}
\sum_{i,j=1}^n
\E\left[
\left|
\Re\left(
\boldsymbol{\beta}_{n,ij}^\top
\boldsymbol{\Xi}_{n,ij}
\right)
\right|^2
\mathbf{1}_{\left\{
\left|
\Re\left(
\boldsymbol{\beta}_{n,ij}^\top
\boldsymbol{\Xi}_{n,ij}
\right)
\right|
>
\varepsilon
\right\}}
\,\middle|\,
A_n
\right] \leq
\frac{
\E\left[
\left\|\boldsymbol{\Xi}\right\|_2^3
\right]
}{\varepsilon}
\max_{1\leq i,j\leq n}
\left\|
\boldsymbol{\beta}_{n,ij}
\right\|_2
\sum_{i,j=1}^n
\left\|
\boldsymbol{\beta}_{n,ij}
\right\|_2^2
\overset{\mathbb{P}}{\rightarrow}
0.
\end{aligned}
\]
The conditional variance of \(L_n\) is
\[
\Var\left(
L_n
\,\middle|\,
A_n
\right)
=
\frac{1}{2}
\sum_{a,b=1}^K
\sum_{\ell,m=1}^r
\overline{
t_{ab}^{\left(\ell\right)}
}
\boldsymbol{\Sigma}_{\ell m}
t_{ab}^{\left(m\right)}.
\]
This is the variance of the corresponding real linear functional of
\(
\left(
G_K^{\left(1\right)},
\ldots,
G_K^{\left(r\right)}
\right).
\)
The conditional Lindeberg--Feller theorem and the Cramér--Wold device therefore imply
\[
d_{\mathrm{BL}}\left(
\mathcal{L}\left(
\left(
\left(C_n^{\left(\ell\right)}\right)_{\left[K\right],\left[K\right]}
\right)_{\ell=1}^r
\,\middle|\,
A_n
\right),
\mathcal{L}\left(
\left(
G_K^{\left(\ell\right)}
\right)_{\ell=1}^r
\right)
\right)
\overset{\mathbb{P}}{\rightarrow}
0.
\]

The vector
\(
\boldsymbol{D}_{n,K}
:=
\left(
D_{n,1},\ldots,D_{n,K}
\right)^\top
\)
is measurable with respect to \(A_n\), and the scalar hard-edge argument gives
\[
\boldsymbol{D}_{n,K}
\overset{\mathrm{d}}{\rightarrow}
\boldsymbol{D}_K
:=
\left(
D_1,\ldots,D_K
\right)^\top.
\]
The same conditioning argument used in the proof of
\eqref{eq:finite-rotated-block-limit} consequently yields
\begin{equation}
\left(
\boldsymbol{D}_{n,K},
\left(
\left(C_n^{\left(\ell\right)}\right)_{\left[K\right],\left[K\right]}
\right)_{\ell=1}^r
\right)
\overset{\mathrm{d}}{\rightarrow}
\left(
\boldsymbol{D}_K,
\left(
G_K^{\left(\ell\right)}
\right)_{\ell=1}^r
\right),
\label{eq:multivariate-finite-block-limit}
\end{equation}
where the Gaussian blocks on the right-hand side are jointly independent of
\(\boldsymbol{D}_K\).

We also record the principal-minor identity needed below. For
\(\boldsymbol{z},\boldsymbol{w}\in\mathbb{C}^r\),
\[
\E\left[
\left(B_n\left(\boldsymbol{z}\right)\right)_{ij}
\overline{
\left(B_n\left(\boldsymbol{w}\right)\right)_{kq}
}
\right]
=
\delta_{ik}\delta_{jq}
\boldsymbol{z}^\top
\boldsymbol{\Sigma}
\overline{\boldsymbol{w}}.
\]
Repeating the double Cauchy--Binet and entrywise permutation calculation used to prove
\eqref{eq:principal-minor-orthogonality} gives, for all
\(S,T\subseteq\left[n\right]\),
\begin{equation}
\begin{aligned}
&\E\left[
\det\left(
\left(C_n\left(\boldsymbol{z}\right)\right)_{S,S}
\right)
\overline{
\det\left(
\left(C_n\left(\boldsymbol{w}\right)\right)_{T,T}
\right)
}
\,\middle|\,
A_n
\right] =
\left|S\right|!
\left(
\boldsymbol{z}^\top
\boldsymbol{\Sigma}
\overline{\boldsymbol{w}}
\right)^{\left|S\right|}
\mathbf{1}_{\left\{S=T\right\}}.
\end{aligned}
\label{eq:multivariate-principal-minor-orthogonality}
\end{equation}

For \(1\leq K\leq n\), define
\[
Q_n^{[K]}\left(\boldsymbol{z}\right)
:=
\det\left(
I_K
-
\operatorname{diag}\left(D_{n,1},\ldots,D_{n,K}\right)
\sum_{\ell=1}^r
z_\ell
\left(C_n^{\left(\ell\right)}\right)_{\left[K\right],\left[K\right]}
\right).
\]
Equivalently, \(Q_n^{[K]}\) is obtained from \(Q_n\) by replacing
\(D_{n,j}\) with zero for \(j>K\).

Expanding both determinants in principal minors and applying
\eqref{eq:multivariate-principal-minor-orthogonality} with
\(\boldsymbol{w}=\boldsymbol{z}\) gives
\begin{equation}
\begin{aligned}
&\E\left[
\left|
Q_n\left(\boldsymbol{z}\right)
-
Q_n^{[K]}\left(\boldsymbol{z}\right)
\right|^2
\,\middle|\,
A_n
\right] =
\sum_{\substack{
S\subseteq\left[n\right]
S\nsubseteq\left[K\right]
}}
\left|S\right|!
\left(
\boldsymbol{z}^{\top}
\boldsymbol{\Sigma}
\overline{\boldsymbol{z}}
\right)^{\left|S\right|}
\prod_{j\in S}D_{n,j}^2.
\end{aligned}
\label{eq:multivariate-conditional-determinant-tail}
\end{equation}
The covariance form is nonnegative because
\[
\boldsymbol{z}^{\top}
\boldsymbol{\Sigma}
\overline{\boldsymbol{z}}
=
\left\langle
\overline{\boldsymbol{z}}|
\boldsymbol{\Sigma}
\overline{\boldsymbol{z}}
\right\rangle.
\]

Fix \(0<R<\rho\). For any
\(f\in\operatorname{Hol}\left(\mathbb{C}^r\right)\), the iterated Cauchy formula gives
\[
\begin{aligned}
f\left(\boldsymbol{z}\right)
&=
\frac{1}{\left(2\pi\right)^r}
\int_{\left[0,2\pi\right]^r}
f\left(
\rho e^{\mathrm{i}\theta_1},
\ldots,
\rho e^{\mathrm{i}\theta_r}
\right)
\prod_{\ell=1}^r
\frac{
\rho e^{\mathrm{i}\theta_\ell}
}{
\rho e^{\mathrm{i}\theta_\ell}-z_\ell
}
\prod_{\ell=1}^r
\mathrm{d}\theta_\ell .
\end{aligned}
\]
If \(\left\|\boldsymbol{z}\right\|_\infty\leq R\), then each factor in the product is bounded in modulus by
\(\rho/\left(\rho-R\right)\). The Cauchy--Schwarz inequality therefore yields
\[
\begin{aligned}
\sup_{\left\|\boldsymbol{z}\right\|_\infty\leq R}
\left|
f\left(\boldsymbol{z}\right)
\right|^2
&\leq
\left(
\frac{\rho}{\rho-R}
\right)^{2r}
\frac{1}{\left(2\pi\right)^r}
\int_{\left[0,2\pi\right]^r}
\left|
f\left(
\rho e^{\mathrm{i}\theta_1},
\ldots,
\rho e^{\mathrm{i}\theta_r}
\right)
\right|^2
\prod_{\ell=1}^r
\mathrm{d}\theta_\ell .
\end{aligned}
\]

On the distinguished boundary of the polydisc,
\[
\left\|
\left(
\rho e^{\mathrm{i}\theta_1},
\ldots,
\rho e^{\mathrm{i}\theta_r}
\right)
\right\|_2^2
=
r\rho^2.
\]
Hence
\[
\boldsymbol{z}^{\top}
\boldsymbol{\Sigma}
\overline{\boldsymbol{z}}
\leq
\left\|
\boldsymbol{\Sigma}
\right\|_{\mathrm{op}}
\left\|
\boldsymbol{z}
\right\|_2^2
\leq
r\rho^2
\left\|
\boldsymbol{\Sigma}
\right\|_{\mathrm{op}}.
\]
Set
\[
t
:=
r\rho^2
\left\|
\boldsymbol{\Sigma}
\right\|_{\mathrm{op}}.
\]
Applying the preceding polydisc estimate to
\(Q_n-Q_n^{[K]}\), and then using
\eqref{eq:multivariate-conditional-determinant-tail}, gives
\[
\begin{aligned}
&\E\left[
\sup_{\left\|\boldsymbol{z}\right\|_\infty\leq R}
\left|
Q_n\left(\boldsymbol{z}\right)
-
Q_n^{[K]}\left(\boldsymbol{z}\right)
\right|^2
\,\middle|\,
A_n
\right]
\leq
\left(
\frac{\rho}{\rho-R}
\right)^{2r}
\sum_{\substack{
S\subseteq\left[n\right]\\
S\nsubseteq\left[K\right]
}}
\left|S\right|!
t^{\left|S\right|}
\prod_{j\in S}D_{n,j}^2.
\end{aligned}
\]

For \(J\leq K\), write
\[
\Delta_{n,J}
:=
\sum_{j>J}D_{n,j}^2,
\qquad
e_a^{\left(n,J\right)}
:=
e_a\left(
D_{n,1}^2,\ldots,D_{n,J}^2
\right).
\]
The remaining subset calculation is identical to the scalar argument leading to
\eqref{eq:compact-determinant-tail-bound}. Indeed, split
\(S=A\mathbin{\dot{\cup}}T\), where
\(A\subseteq\left[J\right]\) and
\(T\subseteq\left\{J+1,\ldots,n\right\}\), and distinguish one index of \(T\) lying above \(K\). If
\(t\Delta_{n,J}<1\), this gives
\begin{equation}
\begin{aligned}
\E\left[
\sup_{\left\|\boldsymbol{z}\right\|_\infty\leq R}
\left|
Q_n\left(\boldsymbol{z}\right)
-
Q_n^{[K]}\left(\boldsymbol{z}\right)
\right|^2
\,\middle|\,
A_n
\right] 
\leq
\left(
\frac{\rho}{\rho-R}
\right)^{2r}
t\Delta_{n,K}
\sum_{a=0}^J
\left(a+1\right)!
e_a^{\left(n,J\right)}
t^a
\left(
1-t\Delta_{n,J}
\right)^{-a-2}.
\end{aligned}
\label{eq:multivariate-compact-tail-bound}
\end{equation}

Fix \(M<\infty\) and consider the event
\[
\left\{
t\Delta_{n,J}\leq\frac{1}{2},
\qquad
\max_{1\leq j\leq J}D_{n,j}\leq M
\right\}.
\]
On this event,
\[
e_a^{\left(n,J\right)}
\leq
\binom{J}{a}M^{2a}
\]
and
\[
\left(
1-t\Delta_{n,J}
\right)^{-a-2}
\leq
2^{a+2}.
\]
Thus, for fixed \(J\), \(M\), and \(t\), the sum on the right-hand side of
\eqref{eq:multivariate-compact-tail-bound} is bounded by a deterministic finite constant. Conditional Markov's inequality, the tightness of the first \(J\) coordinates, and
\eqref{eq:inverse-hard-edge-tail} therefore imply that, for every
\(R,\varepsilon>0\),
\begin{equation}
\lim_{K\to\infty}
\limsup_{n\to\infty}
\mathbb{P}\left(
\sup_{\left\|\boldsymbol{z}\right\|_\infty\leq R}
\left|
Q_n\left(\boldsymbol{z}\right)
-
Q_n^{[K]}\left(\boldsymbol{z}\right)
\right|
>
\varepsilon
\right)
=
0.
\label{eq:multivariate-finite-n-truncation}
\end{equation}
No new singular-value estimate is required here: the quantities
\(\Delta_{n,K}\) are the same inverse hard-edge tails used in the scalar proof.

We next construct the limiting field. Recall that
\[
\boldsymbol{D}
=
\left(D_1,D_2,\ldots\right)
\in\ell^2
\qquad
\text{almost surely},
\]
and write
\[
\Delta_K
:=
\sum_{j>K}D_j^2.
\]
For \(L>K\), regard
\(Q_{\boldsymbol{\Sigma}}^{[K]}\) as an \(L\times L\) determinant by setting
\(D_{K+1},\ldots,D_L\) equal to zero. Conditional on
\(\boldsymbol{D}\), the Gaussian principal minors satisfy
\[
\begin{aligned}
&\E\left[
\det\left(
\left(
\sum_{\ell=1}^r
z_\ell G_L^{\left(\ell\right)}
\right)_{S,S}
\right)
\overline{
\det\left(
\left(
\sum_{\ell=1}^r
w_\ell G_L^{\left(\ell\right)}
\right)_{T,T}
\right)
}
\,\middle|\,
\boldsymbol{D}
\right] 
=
\left|S\right|!
\left(
\boldsymbol{z}^{\top}
\boldsymbol{\Sigma}
\overline{\boldsymbol{w}}
\right)^{\left|S\right|}
\mathbf{1}_{\left\{S=T\right\}}.
\end{aligned}
\]
The same polydisc Cauchy estimate and the same subset decomposition therefore yield, whenever
\(t\Delta_J<1\),
\[
\begin{aligned}
&\E\left[
\sup_{\left\|\boldsymbol{z}\right\|_\infty\leq R}
\left|
Q_{\boldsymbol{\Sigma}}^{[L]}
\left(\boldsymbol{z}\right)
-
Q_{\boldsymbol{\Sigma}}^{[K]}
\left(\boldsymbol{z}\right)
\right|^2
\,\middle|\,
\boldsymbol{D}
\right] \leq
\left(
\frac{\rho}{\rho-R}
\right)^{2r}
t\Delta_K
\sum_{a=0}^J
\left(a+1\right)!
e_a\left(
D_1^2,\ldots,D_J^2
\right)
t^a
\left(
1-t\Delta_J
\right)^{-a-2}.
\end{aligned}
\]
Since \(\Delta_J\to 0\) almost surely, the same localization argument gives
\begin{equation}
\lim_{K\to\infty}
\sup_{L\geq K}
\mathbb{P}\left(
\sup_{\left\|\boldsymbol{z}\right\|_\infty\leq R}
\left|
Q_{\boldsymbol{\Sigma}}^{[L]}
\left(\boldsymbol{z}\right)
-
Q_{\boldsymbol{\Sigma}}^{[K]}
\left(\boldsymbol{z}\right)
\right|
>
\varepsilon
\right)
=
0
\label{eq:multivariate-limit-truncation}
\end{equation}
for every \(R,\varepsilon>0\).

The closed polydiscs
\[
\left\{
\boldsymbol{z}\in\mathbb{C}^r:
\left\|\boldsymbol{z}\right\|_\infty\leq m
\right\},
\qquad
m\geq 1,
\]
exhaust \(\mathbb{C}^r\). Repeating the subsequence and Borel--Cantelli construction used after
\eqref{eq:limiting-compact-truncation}, now on these polydiscs, produces a random function
\[
Q_{\boldsymbol{\Sigma}}
\in
\operatorname{Hol}\left(\mathbb{C}^r\right)
\]
such that
\[
Q_{\boldsymbol{\Sigma}}^{[K]}
\overset{\mathbb{P}}{\rightarrow}
Q_{\boldsymbol{\Sigma}}
\text{ in}
\operatorname{Hol}\left(\mathbb{C}^r\right).
\]
Since
\(Q_{\boldsymbol{\Sigma}}^{[K]}\left(\boldsymbol{0}\right)=1\) for every \(K\), compact-uniform convergence gives
\[
Q_{\boldsymbol{\Sigma}}\left(\boldsymbol{0}\right)=1
\text{ almost surely}.
\]

For each fixed \(K\), the joint convergence in
\eqref{eq:multivariate-finite-block-limit} and the continuous mapping theorem yield
\[
Q_n^{[K]}
\overset{\mathrm{d}}{\rightarrow}
Q_{\boldsymbol{\Sigma}}^{[K]}
\text{ in}
\operatorname{Hol}\left(\mathbb{C}^r\right).
\]
Combining this fixed-\(K\) convergence with
\eqref{eq:multivariate-finite-n-truncation} and
\eqref{eq:multivariate-limit-truncation}, the converging-together argument gives
\[
Q_n
\overset{\mathrm{d}}{\rightarrow}
Q_{\boldsymbol{\Sigma}}
\text{ in}
\operatorname{Hol}\left(\mathbb{C}^r\right).
\]

Fix a nonzero vector \(\boldsymbol{v}\in\mathbb{C}^r\). For each \(K\geq 1\), define
\[
H_K\left(\boldsymbol{v}\right)
:=
\sum_{\ell=1}^r
v_\ell G_K^{\left(\ell\right)}.
\]
For \(1\leq a,b,c,d\leq K\), the covariance assumptions on the Gaussian array give
\[
\begin{aligned}
\E\left[
\left(H_K\left(\boldsymbol{v}\right)\right)_{ab}
\overline{
\left(H_K\left(\boldsymbol{v}\right)\right)_{cd}
}
\right]
=
\sum_{\ell,m=1}^r
v_\ell
\overline{v_m}
\E\left[
G_{ab}^{\left(\ell\right)}
\overline{
G_{cd}^{\left(m\right)}
}
\right] 
=
\delta_{ac}\delta_{bd}
\boldsymbol{v}^{\top}
\boldsymbol{\Sigma}
\overline{\boldsymbol{v}} 
=
\delta_{ac}\delta_{bd}
\sigma_{\boldsymbol{v}}^2.
\end{aligned}
\]
Similarly,
\[
\begin{aligned}
\E\left[
\left(H_K\left(\boldsymbol{v}\right)\right)_{ab}
\left(H_K\left(\boldsymbol{v}\right)\right)_{cd}
\right]
=
\sum_{\ell,m=1}^r
v_\ell v_m
\E\left[
G_{ab}^{\left(\ell\right)}
G_{cd}^{\left(m\right)}
\right] 
=0.
\end{aligned}
\]
The vectors \(\boldsymbol{G}_{ab}\) are independent over the index pairs
\(\left(a,b\right)\). Consequently, the entries of
\(H_K\left(\boldsymbol{v}\right)\) are independent circular complex Gaussian variables with variance
\(\sigma_{\boldsymbol{v}}^2\). Since \(\boldsymbol{\Sigma}\) is positive definite and
\(\boldsymbol{v}\neq\boldsymbol{0}\), we have
\(\sigma_{\boldsymbol{v}}>0\), and therefore
\[
H_K\left(\boldsymbol{v}\right)
\overset{\mathrm{d}}{=}
\sigma_{\boldsymbol{v}}G_K.
\]
The matrix \(H_K\left(\boldsymbol{v}\right)\) is also independent of
\(\left(D_1,\ldots,D_K\right)\). It follows that, as random elements of
\(\operatorname{Hol}\left(\mathbb{C}\right)\),
\begin{align}
Q_{\boldsymbol{\Sigma}}^{[K]}
\left(\,\cdot\,\boldsymbol{v}\right)
&=
\det\left(
I_K
-
\left(\,\cdot\,\right)
\operatorname{diag}\left(D_1,\ldots,D_K\right)
H_K\left(\boldsymbol{v}\right)
\right) \notag\\
&\overset{\mathrm{d}}{=}
\det\left(
I_K
-
\sigma_{\boldsymbol{v}}
\left(\,\cdot\,\right)
\operatorname{diag}\left(D_1,\ldots,D_K\right)
G_K
\right) \notag\\
&=
Q^{[K]}
\left(
\sigma_{\boldsymbol{v}}\,\cdot
\right).
\label{eq:finite-line-restriction}
\end{align}

The restriction map
\[
\mathcal{R}_{\boldsymbol{v}}
\colon
\operatorname{Hol}\left(\mathbb{C}^r\right)
\rightarrow
\operatorname{Hol}\left(\mathbb{C}\right),
\qquad
\mathcal{R}_{\boldsymbol{v}}f\left(t\right)
:=
f\left(t\boldsymbol{v}\right),
\]
is continuous for the compact-open topologies. Indeed, for every compact
\(K_0\subset\mathbb{C}\), the set
\(
\left\{
t\boldsymbol{v}:t\in K_0
\right\}
\)
is compact in \(\mathbb{C}^r\). Uniform convergence on compact subsets of
\(\mathbb{C}^r\) therefore implies uniform convergence after restriction to
\(K_0\).

The compact-uniform convergence proved above gives
\[
Q_{\boldsymbol{\Sigma}}^{[K]}
\left(\,\cdot\,\boldsymbol{v}\right)
\overset{\mathbb{P}}{\rightarrow}
Q_{\boldsymbol{\Sigma}}
\left(\,\cdot\,\boldsymbol{v}\right)
\text{in}
\operatorname{Hol}\left(\mathbb{C}\right).
\]
Theorem~\ref{thm:determinant-field} similarly gives
\[
Q^{[K]}
\left(
\sigma_{\boldsymbol{v}}\,\cdot
\right)
\overset{\mathbb{P}}{\rightarrow}
Q\left(
\sigma_{\boldsymbol{v}}\,\cdot
\right)
\text{in}
\operatorname{Hol}\left(\mathbb{C}\right),
\]
because composition with the dilation
\(t\mapsto\sigma_{\boldsymbol{v}}t\) is continuous in the compact-open topology.
The two sequences in these displays have the same law for every \(K\) by
\eqref{eq:finite-line-restriction}. Uniqueness of weak limits therefore yields
\[
Q_{\boldsymbol{\Sigma}}
\left(\,\cdot\,\boldsymbol{v}\right)
\overset{\mathrm{d}}{=}
Q\left(
\sigma_{\boldsymbol{v}}\,\cdot
\right)
\text{in}
\operatorname{Hol}\left(\mathbb{C}\right),
\]
which proves \eqref{eq:line-restriction-law}.

Applying the same continuous restriction map to
\eqref{eq:multivariate-field-convergence} gives
\[
Q_n\left(\,\cdot\,\boldsymbol{v}\right)
\overset{\mathrm{d}}{\rightarrow}
Q_{\boldsymbol{\Sigma}}
\left(\,\cdot\,\boldsymbol{v}\right)
\text{in}
\operatorname{Hol}\left(\mathbb{C}\right).
\]
Since
\[
Q_{\boldsymbol{\Sigma}}
\left(\boldsymbol{0}\right)=1
\qquad
\text{almost surely},
\]
the limiting line restriction is almost surely not identically zero. Continuity of the zero-counting map at nonzero entire functions, used in the proof of
Theorem~\ref{thm:microscopic-spectrum}, consequently gives
\[
\mathcal{Z}\left(
Q_n\left(\,\cdot\,\boldsymbol{v}\right)
\right)
\overset{\mathrm{d}}{\rightarrow}
\mathcal{Z}\left(
Q_{\boldsymbol{\Sigma}}
\left(\,\cdot\,\boldsymbol{v}\right)
\right)
\qquad
\text{in }
\mathcal{N}\left(\mathbb{C}\right).
\]

If \(z\) is a zero of \(Q\) of multiplicity \(m\), then
\(z/\sigma_{\boldsymbol{v}}\) is a zero of
\(t\mapsto Q\left(\sigma_{\boldsymbol{v}}t\right)\) with the same multiplicity. Hence
\[
\mathcal{Z}\left(
Q\left(
\sigma_{\boldsymbol{v}}\,\cdot
\right)
\right)
=
\sum_{\substack{z\in\mathbb{C}\\Q\left(z\right)=0}}
m_Q\left(z\right)
\delta_{z/\sigma_{\boldsymbol{v}}}.
\]
Theorem~\ref{thm:microscopic-spectrum}\textup{(i)} identifies
\(\mathcal{Z}\left(Q\right)\) in distribution with
\(\operatorname{Gin}_{\infty}\). Combining the preceding displays therefore gives
\[
\mathcal{Z}\left(
Q_n\left(\,\cdot\,\boldsymbol{v}\right)
\right)
\overset{\mathrm{d}}{\rightarrow}
\sum_{Z\in\operatorname{Gin}_{\infty}}
\delta_{Z/\sigma_{\boldsymbol{v}}}.
\]
Equivalently, after rescaling the line parameter by
\(\sigma_{\boldsymbol{v}}\),
\[
\sum_{\substack{t\in\mathbb{C}\\
Q_n\left(t\boldsymbol{v}\right)=0}}
m_{Q_n\left(\,\cdot\,\boldsymbol{v}\right)}\left(t\right)
\delta_{\sigma_{\boldsymbol{v}}t}
\overset{\mathrm{d}}{\rightarrow}
\operatorname{Gin}_{\infty}.
\]

For \(r\geq 2\), the full zero set
\(
\left\{
\boldsymbol{z}\in\mathbb{C}^r:
Q_{\boldsymbol{\Sigma}}\left(\boldsymbol{z}\right)=0
\right\}
\)
is a proper analytic hypersurface, possibly with singular points. Choosing any one fixed nonzero
\(\boldsymbol{v}\) and using the line-restriction law shows that this hypersurface is nonempty almost surely. The result above describes its intersection with each fixed deterministic complex line. 
\end{proof}

\appendix

\section{Normalization of the hard-edge and local-law inputs}
\label{app:hard-edge-inputs}

This appendix records the conventions needed to apply the three hard-edge results used in the proof of Theorem~\ref{thm:determinant-field}. We order singular values increasingly,
\[
s_1\left(T\right)\leq\cdots\leq s_n\left(T\right),
\]
whereas \citet{TaoVu2010} use the decreasing order
\(\sigma_1\left(T\right)\geq\cdots\geq\sigma_n\left(T\right)\). Thus,
\[
s_j\left(T\right)=\sigma_{n-j+1}\left(T\right),
\qquad
1\leq j\leq n.
\]
The matrices \(A_n\) and \(B_n\) have entries of variance one. Whenever a cited local law uses entries of variance \(n^{-1}\), we apply it after dividing the relevant matrix by \(\sqrt{n}\).

For a complex random variable \(\Xi=X+\mathrm{i}Y\), the complex normalization in \citet[Section~1, p.~260]{TaoVu2010} is
\[
\E\left[\Xi\right]=0,
\qquad
\E\left[X^2\right]=\E\left[Y^2\right]=\frac{1}{2},
\qquad
\E\left[XY\right]=0.
\]
Under \(\E\left[\left|\Xi\right|^2\right]=1\), this is equivalent to \(\E\left[\Xi^2\right]=0\). Hence \(\Xi_A\) satisfies their complex normalization.

For \(1\leq k\leq n\), write
\[
\boldsymbol{\Lambda}_{k,n}
:=
\left(
n s_1\left(A_n\right)^2,
\ldots,
n s_k\left(A_n\right)^2
\right)
=
\left(
n\sigma_n\left(A_n\right)^2,
\ldots,
n\sigma_{n-k+1}\left(A_n\right)^2
\right),
\]
and let \(\boldsymbol{\Lambda}_{k,n}^{\mathrm{Gin}}\) denote the corresponding vector for a raw variance-one complex Ginibre matrix. Assumption~\ref{ass:matrix-atoms} supplies the sufficiently high moment required by \citet[Theorem~6.2, pp.~278--279]{TaoVu2010}. That theorem gives a constant \(c>0\) such that, for every \(1\leq k\leq n^c\) and every Borel set \(\Omega\subset\mathbb{R}_+^k\),
\[
\begin{aligned}
&\Pr\left(
\boldsymbol{\Lambda}_{k,n}^{\mathrm{Gin}}
\in
\Omega\setminus\partial_{n^{-c}}\Omega
\right)
-n^{-c}
\leq
\Pr\left(
\boldsymbol{\Lambda}_{k,n}
\in\Omega
\right) \leq
\Pr\left(
\boldsymbol{\Lambda}_{k,n}^{\mathrm{Gin}}
\in
\Omega\cup\partial_{n^{-c}}\Omega
\right)
+n^{-c},
\end{aligned}
\]
where \(\partial_\delta\Omega\) is the set of points at \(\ell^\infty\)-distance at most \(\delta\) from \(\partial\Omega\). In particular, Theorem~6.2 itself permits a growing number of coordinates; the restriction to fixed \(k\) occurs only when the authors subsequently discuss an explicit limiting distribution.

The coupling used in the proof follows directly from this comparison. Put \(\delta_n=n^{-c}\). If \(F\subset\mathbb{R}_+^k\) is closed and
\[
F^r
:=
\left\{
\boldsymbol{x}\in\mathbb{R}_+^k:
\operatorname{dist}_\infty\left(\boldsymbol{x},F\right)\leq r
\right\},
\]
then the upper inequality gives
\[
\Pr\left(
\boldsymbol{\Lambda}_{k,n}\in F
\right)
\leq
\Pr\left(
\boldsymbol{\Lambda}_{k,n}^{\mathrm{Gin}}\in F^{\delta_n}
\right)
+\delta_n.
\]
Applying the lower inequality with \(\Omega=F^{2\delta_n}\) gives
\[
\Pr\left(
\boldsymbol{\Lambda}_{k,n}^{\mathrm{Gin}}\in F
\right)
\leq
\Pr\left(
\boldsymbol{\Lambda}_{k,n}\in F^{2\delta_n}
\right)
+\delta_n.
\]
After enlarging the distance by a factor of two, these are the two Prokhorov inequalities. Strassen's theorem therefore provides a coupling satisfying
\[
\Pr\left(
\left\|
\boldsymbol{\Lambda}_{k,n}
-
\boldsymbol{\Lambda}_{k,n}^{\mathrm{Gin}}
\right\|_\infty
>2\delta_n
\right)
\leq
2\delta_n.
\]
This is the form applied with \(k=\lfloor n^\alpha\rfloor\) in the proof of \eqref{eq:inverse-hard-edge-tail}.

For fixed \(k\), Theorem~6.2 and the Gaussian hard-edge limit give
\[
\boldsymbol{\Lambda}_{k,n}
\overset{\mathrm{d}}{\rightarrow}
\left(X_1,\ldots,X_k\right).
\]
The factor \(4n\) appearing in the Bessel-kernel convention of \citet[Theorem~6.3, p.~279]{TaoVu2010} is not part of this definition: both Theorem~6.2 and our hard-edge coordinates use \(n s_j^2\). The exact complex-Gaussian formula in \citet[Theorem~1.1, p.~262]{TaoVu2010} is
\[
\Pr\left(
n s_1\left(G_n^{\mathrm{Gin}}\right)^2\leq t
\right)
=
1-e^{-t},
\qquad
t\geq0.
\]
Only the fixed atom \(\Xi_A\) is used here. The extension discussed in \citet[Section~6.6, pp.~281--282]{TaoVu2010} is not required by the main theorem.

Consider the Hermitian linearization
\[
H_n
:=
\begin{pmatrix}
0&A_n/\sqrt{n}\\
A_n^*/\sqrt{n}&0
\end{pmatrix},
\qquad
R_n\left(z\right)
:=
\left(H_n-zI_{2n}\right)^{-1}.
\]
In the notation of \citet[Theorem~1.1 and equation~\textup{(1.10a)}, p.~3]{AjankiErdosKruger2014}, the ambient dimension is \(N=2n\) and the parameter in assumption \textup{(A1)} may be taken as \(M=n\). For every fixed \(\delta_0\in\left(0,1\right)\), one has \(N^{\delta_0}\leq M\leq N\) for all sufficiently large \(n\). The variance matrix of \(H_n\) is
\[
\mathcal{S}_n
=
\begin{pmatrix}
0&n^{-1}\boldsymbol{1}_n\boldsymbol{1}_n^*\\
n^{-1}\boldsymbol{1}_n\boldsymbol{1}_n^*&0
\end{pmatrix}.
\]
Its nonzero entries equal \(1/n=1/M\), its row sums equal one, and
\[
\operatorname{spec}\left(\mathcal{S}_n\right)
=
\left\{-1,0,1\right\}.
\]
Thus assumptions \textup{(A1)}--\textup{(A3)} in that paper hold, with any fixed \(\rho\in\left(0,1\right)\). On every nonzero-variance entry, division by the square root of the variance leaves a variable with law \(\Xi_A\) or \(\overline{\Xi_A}\); their normalized moment condition \textup{(1.3)} therefore follows from Assumption~\ref{ass:matrix-atoms}.

Let
\[
m_{\mathrm{sc}}\left(z\right)
:=
\frac{-z+\sqrt{z^2-4}}{2},
\]
with the branch satisfying \(\Im m_{\mathrm{sc}}\left(z\right)>0\) on \(\mathbb{C}_+\). Under the preceding substitutions, equation~\textup{(1.10a)} becomes
\[
\max_{1\leq x,y\leq2n}
\left|
\left(R_n\left(z\right)\right)_{xy}
-
m_{\mathrm{sc}}\left(z\right)\delta_{xy}
\right|
\prec
\sqrt{
\frac{
\Im m_{\mathrm{sc}}\left(z\right)
}{n\Im z}
}
+
\frac{1}{n\Im z},
\]
uniformly for \(\left|z\right|\leq10\) and \(\Im z\geq n^{-1+\gamma}\), for every fixed \(\gamma>0\). Taking \(\eta_n=n^{-1+\gamma}\) and restricting \(E\) to \(\left[-1,1\right]\) gives
\[
\max_{1\leq x\leq2n}
\left|
\left(R_n\left(E+\mathrm{i}\eta_n\right)\right)_{xx}
-
m_{\mathrm{sc}}\left(E+\mathrm{i}\eta_n\right)
\right|
\prec
n^{-\gamma/2}+n^{-\gamma}.
\]
To evaluate this estimate at random eigenvalues, take an \(n^{-4}\)-net of \(\left[-1,1\right]\), use the probability form of stochastic domination on the net, and extend it by
\[
\left\|
R_n\left(E+\mathrm{i}\eta_n\right)
-
R_n\left(E'+\mathrm{i}\eta_n\right)
\right\|_{\mathrm{op}}
\leq
\frac{\left|E-E'\right|}{\eta_n^2}.
\]
Since \(n^{-4}\eta_n^{-2}=n^{-2-2\gamma}\), this yields
\[
\sup_{\left|E\right|\leq1}
\max_{1\leq x\leq2n}
\Im\left(R_n\left(E+\mathrm{i}\eta_n\right)\right)_{xx}
\leq C
\]
with probability tending to one faster than any inverse power of \(n\).

Let \(\boldsymbol{u}_{n,a}\) and \(\boldsymbol{v}_{n,a}\) be the left and right singular vectors associated with \(s_a\left(A_n\right)\). The vectors
\[
\boldsymbol{q}_{n,a}^{\pm}
:=
\frac{1}{\sqrt{2}}
\begin{pmatrix}
\boldsymbol{u}_{n,a}\\
\pm\boldsymbol{v}_{n,a}
\end{pmatrix}
\]
are normalized eigenvectors of \(H_n\) with eigenvalues
\(\lambda_{n,a}^{\pm}=\pm s_a\left(A_n\right)/\sqrt{n}\). For fixed \(a\), hard-edge convergence gives \(\lambda_{n,a}^{\pm}=O_{\mathbb{P}}\left(n^{-1}\right)\), so these eigenvalues lie in \(\left[-1,1\right]\) with probability tending to one. If \(\boldsymbol{q}\) is a normalized eigenvector with eigenvalue \(\lambda\), then
\[
\Im\left(R_n\left(\lambda+\mathrm{i}\eta_n\right)\right)_{xx}
\geq
\frac{\left|\boldsymbol{q}\left(x\right)\right|^2}{\eta_n}.
\]
Consequently, for every fixed \(K\) and every \(\varepsilon>0\), choosing \(\gamma<2\varepsilon\) gives
\[
\max_{1\leq a\leq K}
\left\|\boldsymbol{u}_{n,a}\right\|_\infty
+
\max_{1\leq a\leq K}
\left\|\boldsymbol{v}_{n,a}\right\|_\infty
\leq
n^{-1/2+\varepsilon}
\]
with probability tending to one. This is the estimate used in \eqref{eq:fixed-singular-vector-delocalization}.
Apply \citet[Theorem~2.7\textup{(iii)} and equation~\textup{(2.11)}, p.~8]{AltErdosKruger2017} to
\[
X_n:=\frac{A_n}{\sqrt{n}}.
\]
Both dimensions in the cited paper are equal to our \(n\), and the variance profile is
\[
s_{ij}
:=
\E\left[
\left|\left(X_n\right)_{ij}\right|^2
\right]
=
\frac{1}{n}.
\]
Condition \textup{(A)} holds with \(s^*=2\), condition \textup{(C)} follows from the moments of \(\Xi_A\), condition \textup{(E1)} is the square-matrix assumption, and condition \textup{(F1)} holds with one block, \(I_1=\left[n\right]\), \(Z=\left(1\right)\), and source parameter \(\varphi=1\).

For this constant profile, equation~\textup{(2.1)} in that paper reduces to
\[
-\frac{1}{m\left(\zeta\right)}
=
\zeta-
\frac{1}{1+m\left(\zeta\right)},
\]
whose measure is the standard Marchenko--Pastur law with density
\[
\varrho_{\mathrm{MP}}\left(x\right)
=
\frac{1}{2\pi}
\sqrt{\frac{4-x}{x}}
\1_{\left(0,4\right)}\left(x\right).
\]
The increasingly ordered eigenvalues of \(X_nX_n^*\) are
\[
L_{n,j}
:=
\frac{s_j\left(A_n\right)^2}{n},
\qquad
1\leq j\leq n.
\]
Define \(\gamma_{n,j}\) by
\[
\int_0^{\gamma_{n,j}}
\varrho_{\mathrm{MP}}\left(x\right)\,\mathrm{d}x
=
\frac{j}{n}.
\]
Then the index \(i\left(\gamma_{n,j}\right)\) from equation~\textup{(2.7)} of that paper equals \(j\). Since \(\varrho_{\mathrm{MP}}\) is decreasing and remains positive at its median, the density condition in equation~\textup{(2.11)} holds with one fixed positive constant for every \(1\leq j\leq\lfloor n/2\rfloor\). Consequently, for every \(\vartheta>0\) and \(D>0\),
\[
\Pr\left(
\exists\,1\leq j\leq\left\lfloor n/2\right\rfloor:
\left|L_{n,j}-\gamma_{n,j}\right|
>
\frac{n^\vartheta}{n}
\left(
\sqrt{\gamma_{n,j}}+
\frac{1}{n}
\right)
\right)
\leq
C_{\vartheta,D}n^{-D}.
\]

On the interval from zero to the Marchenko--Pastur median,
\(\varrho_{\mathrm{MP}}\left(x\right)\asymp x^{-1/2}\). Therefore,
\[
\gamma_{n,j}
\asymp
\left(\frac{j}{n}\right)^2,
\qquad
1\leq j\leq\left\lfloor n/2\right\rfloor,
\]
and the rigidity error is bounded by
\[
C n^{-2+\vartheta}\left(j+1\right).
\]
Let \(m_n=\lfloor n^\alpha\rfloor\) and choose \(0<\vartheta<\alpha\). Uniformly for \(m_n\leq j\leq\lfloor n/2\rfloor\),
\[
\frac{
n^{-2+\vartheta}\left(j+1\right)
}{
\gamma_{n,j}
}
\leq
C\frac{n^\vartheta}{j}
\rightarrow0.
\]
It follows, with probability at least \(1-C_Dn^{-D}\) for every fixed \(D>0\), that
\[
L_{n,j}
\geq
c\frac{j^2}{n^2},
\qquad
m_n\leq j\leq\left\lfloor n/2\right\rfloor.
\]
Since
\[
D_{n,j}^2
=
\frac{1}{n s_j\left(A_n\right)^2}
=
\frac{1}{n^2L_{n,j}},
\]
we obtain
\[
D_{n,j}^2
\leq
\frac{C}{j^2},
\qquad
m_n\leq j\leq\left\lfloor n/2\right\rfloor.
\]
At \(j=\lfloor n/2\rfloor\), the same estimate gives \(L_{n,j}\geq c\). Monotonicity then gives \(D_{n,j}^2\leq Cn^{-2}\) for \(j>n/2\), and hence
\[
\sum_{j>m_n}D_{n,j}^2
\leq
C\sum_{j>m_n}^{\lfloor n/2\rfloor}\frac{1}{j^2}
+
\frac{C}{n}
\leq
\frac{C}{m_n}
+
\frac{C}{n}
\]
with overwhelming probability. This is the rigidity estimate combined with the growing-window comparison in the proof of \eqref{eq:inverse-hard-edge-tail}.

\section{Uniform hard-edge comparison for triangular arrays}
\label{app:triangular-arrays}

We remove the restriction that the atom laws remain fixed as \(n\) varies. The only part of the preceding proof that requires additional work is the comparison of a growing number of hard-edge singular values. The local laws, the rotated-block central limit theorem, and the determinant truncation are already uniform under common moment bounds.

For \(m\leq n\), define
\[
\boldsymbol{\Lambda}_{n,m}
:=
\left(
n s_1\left(A_n\right)^2,
\ldots,
n s_m\left(A_n\right)^2
\right),
\]
and let
\[
\boldsymbol{\Lambda}_{n,m}^{\mathrm{Gin}}
:=
\left(
n s_1\left(G_n^{\mathrm{Gin}}\right)^2,
\ldots,
n s_m\left(G_n^{\mathrm{Gin}}\right)^2
\right),
\]
where \(G_n^{\mathrm{Gin}}\) is an \(n\times n\) complex Ginibre matrix with variance-one entries. For a Borel set
\(\Omega\subseteq\left[0,\infty\right)^m\) and \(\delta>0\), write
\[
\partial_\delta\Omega
:=
\left\{
\boldsymbol{x}\in\left[0,\infty\right)^m:
\operatorname{dist}_{\infty}
\left(
\boldsymbol{x},
\partial\Omega
\right)
\leq\delta
\right\}.
\]

\begin{theorem}[Uniform growing-window hard-edge comparison]
\label{thm:triangular-hard-edge}
For each \(p\geq1\), fix \(L_p<\infty\). Let
\(\left(\Xi_n\right)_{n\geq1}\) be complex random variables satisfying
\[
\E\left[\Xi_n\right]=0,
\qquad
\E\left[\left|\Xi_n\right|^2\right]=1,
\qquad
\E\left[\Xi_n^2\right]=0,
\qquad
\sup_{n\geq1}
\E\left[\left|\Xi_n\right|^p\right]
\leq L_p
\]
for every \(p\geq1\). Suppose that \(A_n\) has independent entries with common law
\(\Xi_n\) and is almost surely invertible for every \(n\). This last condition is automatic under the density assumption imposed in the triangular-array application below.

There exist constants \(c_{\mathrm{win}},c_{\mathrm{cmp}}>0\) and
\(n_0<\infty\), depending only on a finite collection of the bounds
\(\left(L_p\right)_{p\geq1}\), such that, for every
\[
1\leq m\leq n^{c_{\mathrm{win}}},
\qquad
n\geq n_0,
\]
and every Borel set
\(\Omega\subseteq\left[0,\infty\right)^m\),
\begin{equation}
\begin{aligned}
&\Pr\left(
\boldsymbol{\Lambda}_{n,m}^{\mathrm{Gin}}
\in
\Omega\setminus
\partial_{n^{-c_{\mathrm{cmp}}}}\Omega
\right)
-
n^{-c_{\mathrm{cmp}}}
\leq
\Pr\left(
\boldsymbol{\Lambda}_{n,m}\in\Omega
\right)
\leq
\Pr\left(
\boldsymbol{\Lambda}_{n,m}^{\mathrm{Gin}}
\in
\Omega\cup
\partial_{n^{-c_{\mathrm{cmp}}}}\Omega
\right)
+
n^{-c_{\mathrm{cmp}}}.
\end{aligned}
\label{eq:triangular-hard-edge-comparison}
\end{equation}
The constants are uniform over all triangular arrays satisfying the displayed moment and circularity conditions.
\end{theorem}

\begin{proof}
The proof follows the sampling and projection argument underlying
\cite[Theorem~6.2, p.~278]{TaoVu2010}, but we keep the law dependence and the growing dimension explicit.

Fix
\[
\theta:=\frac{1}{20},
\qquad
\eta:=\frac{1}{100},
\qquad
c_{\mathrm{nd}}:=\frac{\theta}{4}.
\]
Choose an integer \(C_0\) so large that, with
\(
\alpha:=500/C_0,
\)
\[
\alpha<1-\theta,
\qquad
1-\theta-\frac{20}{C_0}>0,
\qquad
\frac{40}{C_0}<1,
\qquad
\frac{31}{C_0}<\frac{c_{\mathrm{nd}}}{4},
\qquad
11\alpha<\frac{c_{\mathrm{nd}}}{2},
\qquad
2\alpha-1<-\frac{1}{C_0}.
\]
These inequalities respectively control rank interlacing, projection concentration, truncation, the inverse-row union bound, the growing-dimensional Gaussian approximation, and collisions among the sampled indices. The uniform moment assumption gives
\[
\sup_{n\geq1}
\E\left[
\left|\Xi_n\right|^{C_0}
\right]
\leq
L_{C_0}.
\]
We will take
\[
c_{\mathrm{win}}
=
\frac{1}{C_0}
\qquad\text{and}\qquad
c_{\mathrm{cmp}}
=
\frac{1}{2C_0}.
\]
Whenever an estimate holds with overwhelming probability, we invoke it with the fixed exponent \(D_*=10\). Its failure probability is then negligible relative to \(n^{-1/C_0}\), which is the largest exceptional probability retained below.

\emph{Uniform truncation and normalization.}
Set
\[
\tau_n:=n^{10/C_0}.
\]
The union bound and Markov's inequality give
\[
\Pr\left(
\max_{1\leq i,j\leq n}
\left|
\left(A_n\right)_{ij}
\right|
>
\tau_n
\right)
\leq
n^2L_{C_0}\tau_n^{-C_0}
=
L_{C_0}n^{-8}.
\]

Let
\[
Y_n
:=
\Xi_n
\mathbf{1}_{\left\{
\left|\Xi_n\right|\leq\tau_n
\right\}},
\]
and regard \(Y_n\) as a vector in \(\mathbb{R}^2\). Write
\[
\boldsymbol{\mu}_n
:=
\E\left[
\begin{pmatrix}
\Re Y_n\\
\Im Y_n
\end{pmatrix}
\right]
\]
and
\[
S_n
:=
\operatorname{Cov}\left(
\begin{pmatrix}
\Re Y_n\\
\Im Y_n
\end{pmatrix}
\right).
\]
The circular second-moment assumptions on \(\Xi_n\) imply that the untruncated covariance matrix is
\(\frac{1}{2}I_2\). Moreover,
\[
\left\|
\boldsymbol{\mu}_n
\right\|_2
\leq
L_{C_0}\tau_n^{1-C_0}
\]
and
\[
\left\|
S_n-\frac{1}{2}I_2
\right\|_{\mathrm{op}}
\leq
C L_{C_0}\tau_n^{2-C_0}.
\]
For all sufficiently large \(n\), \(S_n\) is positive definite. Define
\[
T_n
:=
\left(
2S_n
\right)^{-1/2}
\]
and let \(\widehat{\Xi}_n\) be the complex random variable corresponding to
\[
T_n
\left[
\begin{pmatrix}
\Re Y_n\\
\Im Y_n
\end{pmatrix}
-
\boldsymbol{\mu}_n
\right].
\]
Then
\[
\E\left[
\widehat{\Xi}_n
\right]=0,
\qquad
\E\left[
\left|
\widehat{\Xi}_n
\right|^2
\right]=1,
\qquad
\E\left[
\widehat{\Xi}_n^2
\right]=0.
\]
Functional calculus for the positive matrix \(2S_n\) gives
\[
\left\|
T_n-I_2
\right\|_{\mathrm{op}}
\leq
C L_{C_0}\tau_n^{2-C_0}.
\]
Consequently, on the event
\(\left|\Xi_n\right|\leq\tau_n\),
\[
\left|
\widehat{\Xi}_n-\Xi_n
\right|
\leq
C L_{C_0}\tau_n^{3-C_0}.
\]

Let \(\widehat{A}_n\) be the matrix obtained by applying this construction entrywise, using the same underlying entries as \(A_n\). On the event that no entry is truncated,
\[
\left\|
A_n-\widehat{A}_n
\right\|_{\mathrm{F}}
\leq
C L_{C_0}n\tau_n^{3-C_0},
\]
while
\[
\left\|
A_n
\right\|_{\mathrm{op}}
+
\left\|
\widehat{A}_n
\right\|_{\mathrm{op}}
\leq
C n\tau_n.
\]
The singular-value perturbation inequality therefore gives, simultaneously for every \(j\),
\[
\begin{aligned}
n
\left|
s_j\left(A_n\right)^2
-
s_j\left(\widehat{A}_n\right)^2
\right|
&\leq
n
\left\|
A_n-\widehat{A}_n
\right\|_{\mathrm{op}}
\left(
\left\|
A_n
\right\|_{\mathrm{op}}
+
\left\|
\widehat{A}_n
\right\|_{\mathrm{op}}
\right)
\\
&\leq
C L_{C_0}
n^3\tau_n^{4-C_0}
\\
&=
C L_{C_0}
n^{-7+40/C_0}.
\end{aligned}
\]
By the choice of \(C_0\), the last quantity is at most \(n^{-6}\). Thus, the truncation and normalization alter the vector of scaled singular values by at most \(n^{-6}\), except on an event of probability at most
\(L_{C_0}n^{-8}\). We may therefore assume throughout the remainder of the proof that
\[
\left|\Xi_n\right|
\leq
2n^{10/C_0}
\]
almost surely, without changing the final comparison exponent.

Put
\[
s:=\left\lfloor n^\alpha\right\rfloor.
\]
Let
\(
\boldsymbol{X}_1,\ldots,\boldsymbol{X}_n
\)
be the columns of \(A_n\), and define
\[
V_{s,n}
:=
\operatorname{span}
\left(
\boldsymbol{X}_{s+1},
\ldots,
\boldsymbol{X}_n
\right)^\perp.
\]
We claim that, for every fixed \(D>0\), there is a constant \(C_D<\infty\) such that
\begin{equation}
\Pr\left(
\sup_{\substack{
\boldsymbol{v}\in V_{s,n}\\
\left\|\boldsymbol{v}\right\|_2=1
}}
\left\|
\boldsymbol{v}
\right\|_\infty
>
n^{-c_{\mathrm{nd}}}
\right)
\leq
C_Dn^{-D}.
\label{eq:triangular-nondegeneracy}
\end{equation}
The exponent \(c_{\mathrm{nd}}\), the constants \(C_D\), and the large-\(n\) thresholds are uniform over the atom array, subject only to the common moment bounds.

To prove the claim, suppose that
\(\boldsymbol{v}\in V_{s,n}\) is a unit vector satisfying
\(\left|v_1\right|\geq n^{-\theta/4}\). Let
\(\mathcal{B}_n\) be the
\(\left(n-s\right)\times n\) matrix whose rows are
\(\boldsymbol{X}_{s+1}^*,\ldots,\boldsymbol{X}_n^*\). Then
\[
\mathcal{B}_n\boldsymbol{v}=\boldsymbol{0}.
\]
Write
\[
\mathcal{B}_n
=
\left(
\boldsymbol{Y}_1,\mathcal{B}_n'
\right)
\qquad\text{and}\qquad
\boldsymbol{v}
=
\begin{pmatrix}
v_1\\
\boldsymbol{v}'
\end{pmatrix}.
\]
Hence
\[
v_1\boldsymbol{Y}_1
+
\mathcal{B}_n'\boldsymbol{v}'
=
\boldsymbol{0}.
\]

Let \(N:=n-s\), and take an \(N\times N\) square block
\(\mathcal{H}_N\) of \(\mathcal{B}_n'\). Apply
\citet[Theorem~2.7\textup{(iii)} and equation~\textup{(2.11)}]{AltErdosKruger2017}
to \(\mathcal{H}_N/\sqrt{N}\). The variance profile is constant, the matrix is square and fully indecomposable, and the common moment bounds make every model parameter uniform in \(n\). If
\[
d_N
:=
\left\lfloor
N^{1-\theta}
\right\rfloor,
\]
the corresponding Marchenko--Pastur quantile is of order
\(N^{-2\theta}\). The rigidity estimate therefore gives
\[
s_{d_N}\left(\mathcal{H}_N\right)
\leq
C N^{1/2-\theta}
\]
with overwhelming probability.

The matrix \(\mathcal{B}_n'\) is obtained from
\(\mathcal{H}_N\) by adjoining \(s-1\) columns. The difference between their Gram matrices has rank at most \(s-1\). Since
\(s=o\left(d_N\right)\), rank interlacing implies that
\(\mathcal{B}_n'\) has at least \(d_N/2\) singular values not exceeding
\(C n^{1/2-\theta}\), with overwhelming probability. Let \(W_n\) be the span of the corresponding left singular vectors. Then
\[
\dim\left(W_n\right)
\geq
\frac{d_N}{2}
\]
and
\[
\left\|
\pi_{W_n}\mathcal{B}_n'
\right\|_{\mathrm{op}}
\leq
C n^{1/2-\theta}.
\]
Conditional on \(\mathcal{B}_n'\), the vector
\(\boldsymbol{Y}_1\) is independent of \(W_n\). Applying
\citet[Lemma~4.2, pp.~275 and 290--291]{TaoVu2010}, with
\[
K=2n^{10/C_0}
\qquad\text{and}\qquad
d=\dim\left(W_n\right),
\]
gives
\[
\Pr\left(
\left\|
\pi_{W_n}\boldsymbol{Y}_1
\right\|_2
<
\frac{1}{2}
\sqrt{\dim\left(W_n\right)}
\,\middle|\,
\mathcal{B}_n'
\right)
\leq
\exp\left(
-c n^{1-\theta-20/C_0}
\right).
\]
On the complementary event,
\[
\left\|
\pi_{W_n}\boldsymbol{Y}_1
\right\|_2
\geq
c n^{1/2-\theta/2}.
\]
On the other hand, projecting
\(v_1\boldsymbol{Y}_1+\mathcal{B}_n'\boldsymbol{v}'=\boldsymbol{0}\)
onto \(W_n\) gives
\[
\left\|
\pi_{W_n}\boldsymbol{Y}_1
\right\|_2
\leq
\left|v_1\right|^{-1}
\left\|
\pi_{W_n}\mathcal{B}_n'
\right\|_{\mathrm{op}}
\leq
C n^{1/2-\theta+\theta/4}.
\]
The exponent on the last line is
\(1/2-3\theta/4\), which is strictly smaller than
\(1/2-\theta/2\). This is impossible for all sufficiently large \(n\). For a prescribed \(D>0\), invoke the rigidity estimate above with failure probability at most \(C_Dn^{-D-2}\). The conditional projection estimate is stretched exponential because \(1-\theta-20/C_0>0\). A union bound over the \(n\) coordinates therefore gives \eqref{eq:triangular-nondegeneracy} with the fixed choice \(c_{\mathrm{nd}}=\theta/4\).

Let
\(\boldsymbol{R}_1,\ldots,\boldsymbol{R}_n\)
be the rows of \(A_n^{-1}\). We prove
\begin{equation}
\Pr\left(
\max_{1\leq i\leq n}
\left\|
\boldsymbol{R}_i
\right\|_2
>
n^{100/C_0}
\right)
\leq
C n^{-1/C_0}.
\label{eq:triangular-inverse-row-bound}
\end{equation}

Let
\[
d_i
:=
\operatorname{dist}\left(
\boldsymbol{X}_i,
\operatorname{span}
\left(
\boldsymbol{X}_1,\ldots,
\boldsymbol{X}_{i-1},
\boldsymbol{X}_{i+1},\ldots,
\boldsymbol{X}_n
\right)
\right).
\]
By \citet[Lemma~5.1, p.~276]{TaoVu2010},
\[
\left\|
\boldsymbol{R}_i
\right\|_2
=
d_i^{-1}.
\]

Conditional on all columns except \(\boldsymbol{X}_i\), let
\(\boldsymbol{v}\) be a unit normal to their span. The same argument as in \eqref{eq:triangular-nondegeneracy}, with \(s=1\) and \(D=D_*\), gives
\(
\left\|
\boldsymbol{v}
\right\|_\infty
\leq
n^{-c_{\mathrm{nd}}}
\)
outside an event of probability at most \(C n^{-D_*}\). Since
\[
d_i
=
\left|
\sum_{q=1}^n
\overline{v_q}
\left(\boldsymbol{X}_i\right)_q
\right|,
\]
the frame Berry--Esseen estimate
\citep[Proposition~3.4, pp.~271--272, and Proposition~D.2, pp.~288--289]{TaoVu2010}, after absorbing the \(C n^{-D_*}\) exceptional probability into the Berry--Esseen remainder, implies that, for every \(t>0\),
\[
\Pr\left(
d_i\leq t
\right)
\leq
C
\left(
t+n^{-c_{\mathrm{nd}}/4}
\right).
\]
The constant is uniform because
\[
\sup_n
\E\left[
\left|\Xi_n\right|^3
\right]
<\infty.
\]

Set
\[
\ell_n
:=
\left\lfloor
n^{30/C_0}
\right\rfloor
\qquad\text{and}\qquad
t_n
:=
n^{-31/C_0}.
\]
A union bound gives
\[
\Pr\left(
\min_{1\leq i\leq\ell_n}
d_i<t_n
\right)
\leq
C
\left(
n^{-1/C_0}
+
n^{30/C_0-c_{\mathrm{nd}}/4}
\right).
\]
By the choice of \(C_0\), the second term is at most
\(n^{-1/C_0}\).

For \(j>\ell_n\), let \(V_{\ell_n,j}\) be the orthogonal complement of the span of
\(
\boldsymbol{X}_{\ell_n+1},
\ldots,
\boldsymbol{X}_{j-1},
\boldsymbol{X}_{j+1},
\ldots,
\boldsymbol{X}_n,
\)
and let \(\pi_{\ell_n,j}\) denote the corresponding projection. This space has dimension
\(\ell_n+1\). The concentration estimate used above, together with
\[
\frac{\ell_n}{\tau_n^2}
=
n^{10/C_0},
\]
shows that, simultaneously for
\(1\leq i\leq\ell_n\) and \(j>\ell_n\),
\[
c\sqrt{\ell_n}
\leq
\left\|
\pi_{\ell_n,j}\boldsymbol{X}_i
\right\|_2
\leq
C\sqrt{\ell_n}
\]
and
\[
c\sqrt{\ell_n}
\leq
\left\|
\pi_{\ell_n,j}\boldsymbol{X}_j
\right\|_2
\leq
C\sqrt{\ell_n}
\]
outside an event of probability at most
\(\exp\left(-n^{c/C_0}\right)\).

Lemma~5.3 of \citet[p.~277]{TaoVu2010} gives
\[
d_j
\geq
\frac{
\left\|
\pi_{\ell_n,j}\boldsymbol{X}_j
\right\|_2
}{
1+
\displaystyle
\sum_{i=1}^{\ell_n}
\frac{
\left\|
\pi_{\ell_n,j}\boldsymbol{X}_i
\right\|_2
}{
d_i
}
}.
\]
On the event that \(d_i\geq t_n\) for \(i\leq\ell_n\), the preceding projection bounds imply
\[
d_j
\geq
\frac{
c\sqrt{\ell_n}
}{
1+
C\ell_n^{3/2}t_n^{-1}
}
\geq
c\frac{t_n}{\ell_n}
\geq
n^{-61/C_0}
\]
for all sufficiently large \(n\). This proves
\eqref{eq:triangular-inverse-row-bound}.

Every probability and threshold in this argument depends on the atom law only through common moment bounds. The truncation support is deterministic and the concentration inequality is uniform over all product measures supported in the corresponding disc.

Let
\(
\mathcal{I}_1,\ldots,\mathcal{I}_s
\)
be independent uniform indices in
\(\left\{1,\ldots,n\right\}\), independent of \(A_n\), and let
\(\mathsf{B}_n\) be the \(s\times n\) matrix whose \(q\)-th row is
\(\boldsymbol{R}_{\mathcal{I}_q}\). Lemma~3.1 of
\citet[p.~270]{TaoVu2010} gives, conditional on \(A_n\),
\[
\E_{\mathcal{I}}
\left[
\left\|
A_n^{-*}A_n^{-1}
-
\frac{n}{s}
\mathsf{B}_n^*
\mathsf{B}_n
\right\|_{\mathrm{F}}^2
\right]
\leq
\frac{n}{s}
\sum_{i=1}^n
\left\|
\boldsymbol{R}_i
\right\|_2^4.
\]
On the event in
\eqref{eq:triangular-inverse-row-bound}, the right-hand side is bounded by
\[
\frac{n^2}{s}
n^{400/C_0}
\leq
n^{2-100/C_0}.
\]
Markov's inequality therefore gives
\begin{equation}
\left\|
A_n^{-*}A_n^{-1}
-
\frac{n}{s}
\mathsf{B}_n^*
\mathsf{B}_n
\right\|_{\mathrm{F}}
\leq
n^{1-25/C_0}
\label{eq:triangular-sampling-bound}
\end{equation}
except on an event of probability at most
\(n^{-50/C_0}\).

The probability that the sampled indices are not distinct is at most
\(
\frac{s^2}{n}
\leq
n^{2\alpha-1}.
\)
On the complementary event, exchangeability of the rows of \(A_n^{-1}\) permits us to identify the sampled rows, after a random permutation, with the first \(s\) rows. Let
\(
V_{s,n}
=
\operatorname{span}
\left(
\boldsymbol{X}_{s+1},
\ldots,
\boldsymbol{X}_n
\right)^\perp
\)
as above, choose an orthonormal identification of \(V_{s,n}\) with
\(\mathbb{C}^s\), and let
\(
\pi_n
:
\mathbb{C}^n
\rightarrow
\mathbb{C}^s
\)
be the resulting partial isometry. Define the \(s\times s\) matrix
\[
\mathsf{M}_n
:=
\left(
\pi_n\boldsymbol{X}_1,
\ldots,
\pi_n\boldsymbol{X}_s
\right).
\]
Lemma~3.3 of \citet[p.~271]{TaoVu2010} gives
\[
s_j\left(\mathsf{B}_n\right)
=
s_j\left(\mathsf{M}_n\right)^{-1}
\]
when the singular values of \(\mathsf{B}_n\) are read in decreasing order and those of \(\mathsf{M}_n\) increasingly.

For \(1\leq j\leq s\), put
\[
\lambda_{n,j}
:=
n s_j\left(A_n\right)^2
\qquad\text{and}\qquad
\mu_{n,j}
:=
s s_j\left(\mathsf{M}_n\right)^2.
\]
The eigenvalue perturbation inequality applied to
\eqref{eq:triangular-sampling-bound} gives
\begin{equation}
\max_{1\leq j\leq s}
\left|
\lambda_{n,j}^{-1}
-
\mu_{n,j}^{-1}
\right|
\leq
n^{-25/C_0}.
\label{eq:triangular-reciprocal-comparison}
\end{equation}

Conditional on
\(\boldsymbol{X}_{s+1},\ldots,\boldsymbol{X}_n\), the matrix
\(\mathsf{M}_n\) is a linear image of the entries in the first \(s\) columns. Under the Hilbert--Schmidt identification
\(\mathbb{C}^{s\times s}\simeq\mathbb{C}^{s^2}\), it can be written as
\[
\mathsf{M}_n
=
\sum_{i=1}^s
\sum_{j=1}^n
\left(A_n\right)_{ji}
\mathsf{V}_{ij},
\]
where the deterministic matrices
\(\mathsf{V}_{ij}\) satisfy
\[
\sum_{i=1}^s
\sum_{j=1}^n
\mathsf{V}_{ij}
\mathsf{V}_{ij}^*
=
I_{s^2}.
\]
Moreover,
\[
\left\|
\mathsf{V}_{ij}
\right\|_{\mathrm{HS}}
=
\left\|
\pi_n\boldsymbol{e}_j
\right\|_2
=
\sup_{\substack{
\boldsymbol{v}\in V_{s,n}\\
\left\|\boldsymbol{v}\right\|_2=1
}}
\left|v_j\right|.
\]
On the event in
\eqref{eq:triangular-nondegeneracy},
\[
\max_{i,j}
\left\|
\mathsf{V}_{ij}
\right\|_{\mathrm{HS}}
\leq
n^{-c_{\mathrm{nd}}}.
\]

Set
\(
\varepsilon_n
:=
n^{-2\alpha}
\)
and
\[
\delta_n
:=
C s^5\varepsilon_n^{-3}n^{-c_{\mathrm{nd}}}
=
C n^{11\alpha-c_{\mathrm{nd}}}.
\]
Conditional on \(\boldsymbol{X}_{s+1},\ldots,\boldsymbol{X}_n\) and on the event in \eqref{eq:triangular-nondegeneracy}, Proposition~D.2 of \citet[pp.~288--289]{TaoVu2010}, applied in complex dimension \(s^2\), gives the following two-parameter comparison. If \(\mu_n\) is the conditional law of \(\mathsf{M}_n\), \(\gamma_s\) is the law of an \(s\times s\) complex Ginibre matrix, and
\[
F^{\rho}
:=
\left\{
X\in\mathbb{C}^{s\times s}:
\operatorname{dist}_{\max}\left(X,F\right)<\rho
\right\},
\]
then, for every closed set \(F\subseteq\mathbb{C}^{s\times s}\),
\[
\mu_n\left(F\right)
\leq
\gamma_s\left(F^{2\varepsilon_n}\right)+\delta_n,
\qquad
\gamma_s\left(F\right)
\leq
\mu_n\left(F^{2\varepsilon_n}\right)+\delta_n.
\]
The first inequality follows from the upper Berry--Esseen comparison after enlarging the boundary neighborhood; the second follows by applying the lower comparison to \(F^{2\varepsilon_n}\). By the choice of \(C_0\), \(\delta_n\leq Cn^{-c_{\mathrm{nd}}/2}\).

The measurable-kernel form of Strassen's theorem, applicable because the conditional laws take values in a finite-dimensional Polish space, gives a conditional coupling for which the entrywise distance exceeds \(2\varepsilon_n\) with probability at most \(\delta_n\). Gluing these conditional couplings to the law of the final \(n-s\) columns and adding the exceptional probability in \eqref{eq:triangular-nondegeneracy}, invoked with \(D=D_*\), yields a joint realization satisfying
\begin{equation}
\left\|
\mathsf{M}_n
-
G_s^{\mathrm{Gin}}
\right\|_{\max}
\leq
2\varepsilon_n
\label{eq:triangular-entry-coupling}
\end{equation}
except on an event of probability at most
\(C n^{-c_{\mathrm{nd}}/2}\). Here
\(\left\|\cdot\right\|_{\max}\) is the largest entry modulus.

Put
\[
J_s
:=
\left\lceil
s^\eta
\right\rceil.
\]
The rigidity estimate
\citet[Theorem~2.7\textup{(iii)} and equation~\textup{(2.11)}]{AltErdosKruger2017},
applied to \(G_s^{\mathrm{Gin}}/\sqrt{s}\), gives
\begin{equation}
s
s_{J_s}\left(
G_s^{\mathrm{Gin}}
\right)^2
\leq
C s^{2\eta}
\label{eq:triangular-gaussian-window}
\end{equation}
with overwhelming probability.

Indeed, if
\[
L_{s,j}
:=
\frac{
s_j\left(G_s^{\mathrm{Gin}}\right)^2
}{s},
\]
then the Marchenko--Pastur quantile at
\(j=J_s\) is of order
\(
\left(
\frac{J_s}{s}
\right)^2
=
s^{-2+2\eta}.
\)
Taking the tolerance exponent in the rigidity estimate smaller than \(\eta\), its error is of smaller order than this quantile. Thus
\(
L_{s,J_s}
\leq
C s^{-2+2\eta},
\)
which is equivalent to
\eqref{eq:triangular-gaussian-window}.

For
\[
m
\leq
n^{1/C_0},
\]
our choices give
\(
m
\leq
n^{1/C_0}
\leq
n^{5/C_0}
=
s^\eta
\leq
J_s
\)
for all sufficiently large \(n\). Hence
\[
s
s_j\left(
G_s^{\mathrm{Gin}}
\right)^2
\leq
C s^{2\eta}
=
C n^{10/C_0},
\qquad
1\leq j\leq m.
\]

The entrywise coupling
\eqref{eq:triangular-entry-coupling} implies
\[
\left\|
\mathsf{M}_n
-
G_s^{\mathrm{Gin}}
\right\|_{\mathrm{op}}
\leq
2s\varepsilon_n
\leq
C s^{-1}.
\]
Therefore, for \(1\leq j\leq m\),
\[
\begin{aligned}
&
\left|
s s_j\left(\mathsf{M}_n\right)^2
-
s s_j\left(G_s^{\mathrm{Gin}}\right)^2
\right|\leq
s
\left[
2s_j\left(G_s^{\mathrm{Gin}}\right)
\left\|
\mathsf{M}_n-G_s^{\mathrm{Gin}}
\right\|_{\mathrm{op}}
+
\left\|
\mathsf{M}_n-G_s^{\mathrm{Gin}}
\right\|_{\mathrm{op}}^2
\right]
\leq
C
\left(
s^{-1/2+\eta}
+
s^{-1}
\right).
\end{aligned}
\]
In particular,
\begin{equation}
\max_{1\leq j\leq m}
\mu_{n,j}
\leq
C n^{10/C_0}
\label{eq:triangular-mu-upper}
\end{equation}
and
\begin{equation}
\max_{1\leq j\leq m}
\left|
\mu_{n,j}
-
s s_j\left(G_s^{\mathrm{Gin}}\right)^2
\right|
\leq
n^{-c_1}
\label{eq:triangular-projected-direct}
\end{equation}
for some \(c_1>0\) depending only on \(C_0\).

From
\eqref{eq:triangular-reciprocal-comparison} and
\eqref{eq:triangular-mu-upper},
\[
n^{-25/C_0}\mu_{n,j}
\leq
C n^{-15/C_0}
\]
for \(j\leq m\). Thus, for sufficiently large \(n\),
\[
\lambda_{n,j}^{-1}
\geq
\mu_{n,j}^{-1}
-
n^{-25/C_0}
\geq
\frac{1}{2}
\mu_{n,j}^{-1},
\]
and hence
\[
\lambda_{n,j}
\leq
2\mu_{n,j}.
\]
Consequently,
\[
\begin{aligned}
\left|
\lambda_{n,j}
-
\mu_{n,j}
\right|
=
\lambda_{n,j}\mu_{n,j}
\left|
\lambda_{n,j}^{-1}
-
\mu_{n,j}^{-1}
\right|
\leq
2\mu_{n,j}^2
n^{-25/C_0}
\leq
C n^{-5/C_0}.
\end{aligned}
\]
Combining this estimate with
\eqref{eq:triangular-projected-direct} gives a coupling satisfying
\begin{equation}
\max_{1\leq j\leq m}
\left|
n s_j\left(A_n\right)^2
-
s s_j\left(G_s^{\mathrm{Gin}}\right)^2
\right|
\leq
C n^{-5/C_0}
\label{eq:triangular-to-small-gaussian}
\end{equation}
outside an event of probability at most
\[
C n^{-1/C_0}.
\]
The dominant failure probability here is
\eqref{eq:triangular-inverse-row-bound}; every other exceptional probability is either overwhelming or bounded by a strictly larger negative power of \(n\).

\emph{Comparison with the Gaussian matrix of dimension \(n\).}
Apply the same sampling argument to an \(n\times n\) complex Ginibre matrix
\(G_n^{\mathrm{Gin}}\). Conditional on its final \(n-s\) columns, the projected first \(s\) columns form an exact \(s\times s\) Ginibre matrix. Indeed, orthogonal projection onto an \(s\)-dimensional subspace preserves the standard complex Gaussian law, and the projected columns remain independent. Therefore, the preceding proof, without the Berry--Esseen approximation, gives
\begin{equation}
\max_{1\leq j\leq m}
\left|
n s_j\left(G_n^{\mathrm{Gin}}\right)^2
-
s s_j\left(G_s^{\mathrm{Gin}}\right)^2
\right|
\leq
C n^{-5/C_0}
\label{eq:large-to-small-gaussian}
\end{equation}
outside an event of probability at most
\(C n^{-1/C_0}\).

Let \(d_{\mathrm{P}}^{\left(m\right)}\) denote the Prokhorov metric on
\(\left[0,\infty\right)^m\) induced by the
\(\ell^\infty\) norm. Equations
\eqref{eq:triangular-to-small-gaussian} and
\eqref{eq:large-to-small-gaussian}, together with the triangle inequality for the Prokhorov metric, imply
\[
d_{\mathrm{P}}^{\left(m\right)}
\left(
\mathcal{L}\left(
\boldsymbol{\Lambda}_{n,m}
\right),
\mathcal{L}\left(
\boldsymbol{\Lambda}_{n,m}^{\mathrm{Gin}}
\right)
\right)
\leq
C n^{-1/C_0}
\]
uniformly over
\(m\leq n^{1/C_0}\).

After increasing the uniform threshold \(n_0\), the right-hand side is at most
\(
n^{-1/(2C_0)}.
\)
The defining inequalities for the Prokhorov metric then give
\eqref{eq:triangular-hard-edge-comparison} with
\(
c_{\mathrm{win}}
=
\frac{1}{C_0}
\text{ and }
c_{\mathrm{cmp}}
=
\frac{1}{2C_0}.
\)
This completes the proof.
\end{proof}

We now apply Theorem~\ref{thm:triangular-hard-edge} to the matrix ratios considered in the main text. For each \(\nu\in\left\{A,B\right\}\), let the entries of the \(\nu\)-matrix at dimension \(n\) have common law \(\Xi_{\nu,n}\), where
\[
\E\left[\Xi_{\nu,n}\right]=0,
\qquad
\E\left[\left|\Xi_{\nu,n}\right|^2\right]=1,
\qquad
\E\left[\Xi_{\nu,n}^2\right]=0,
\]
\[
\sup_{n\geq1}\E\left[\left|\Xi_{\nu,n}\right|^p\right]<\infty
\qquad
\text{for every }p\geq1,
\]
and the corresponding densities are bounded uniformly in \(n\). Under these scalar triangular-array assumptions, Theorem~\ref{thm:triangular-hard-edge} supplies the growing hard-edge comparison used in the proof of \eqref{eq:inverse-hard-edge-tail}. The square-Gram rigidity estimate is uniform because its constants depend only on the common moment bounds and the constant variance profile. The imprimitive local law used for singular-vector delocalization has the same uniformity. The conditional Lyapunov bounds for the rotated blocks use only a common third-moment bound, while principal-minor orthogonality remains exact at every \(n\). Consequently, Theorems~\ref{thm:determinant-field} and~\ref{thm:microscopic-spectrum} remain valid when the scalar atom laws vary with \(n\) under these conditions.

For the multivariate conclusion, retain the preceding triangular-array assumptions for \(A_n\), replace the scalar perturbation law by a vector-valued law, fix \(r\geq1\), and fix a Hermitian positive-definite matrix \(\boldsymbol{\Sigma}\in\mathbb{C}^{r\times r}\). At dimension \(n\), let
\[
\boldsymbol{\Xi}_{B,n,ij}
:=
\left(
\left(B_n^{\left(1\right)}\right)_{ij},
\ldots,
\left(B_n^{\left(r\right)}\right)_{ij}
\right)^\top,
\qquad
1\leq i,j\leq n,
\]
be independent and identically distributed copies of a vector \(\boldsymbol{\Xi}_{B,n}\in\mathbb{C}^r\), independent of \(A_n\), and assume
\[
\E\left[\boldsymbol{\Xi}_{B,n}\right]=\boldsymbol{0},
\qquad
\E\left[\boldsymbol{\Xi}_{B,n}\boldsymbol{\Xi}_{B,n}^*\right]=\boldsymbol{\Sigma},
\qquad
\E\left[\boldsymbol{\Xi}_{B,n}\boldsymbol{\Xi}_{B,n}^\top\right]=\boldsymbol{0}
\]
for every \(n\), together with
\[
\sup_{n\geq1}
\E\left[
\left\|\boldsymbol{\Xi}_{B,n}\right\|_2^p
\right]
<\infty
\qquad
\text{for every }p\geq1.
\]
The covariance matrix is therefore fixed across the triangular array, rather than merely bounded above and below. Conditional on \(A_n\), every fixed family of rotated blocks has the same covariance \(\boldsymbol{\Sigma}\) and vanishing pseudocovariance as in the proof of Theorem~\ref{thm:several-directions}; the common vector moment bounds make the conditional Lindeberg estimates uniform. The principal-minor identity \eqref{eq:multivariate-principal-minor-orthogonality} also remains exact with the same covariance form at each \(n\). Hence Theorem~\ref{thm:several-directions} remains valid for these vector-valued triangular arrays, with limiting field \(Q_{\boldsymbol{\Sigma}}\).

\bibliographystyle{apalike}
\bibliography{hard_edge_determinant_fields_references}
\end{document}